\documentclass[12pt]{amsart}
\usepackage{graphicx,amssymb}
\usepackage[usenames]{color}
\AtEndDocument{\vfill\eject\batchmode}
\newtheorem{theorem}{Theorem}[section]
\newtheorem{lemma}[theorem]{Lemma}
\newtheorem{proposition}[theorem]{Proposition}
\newtheorem{corollary}[theorem]{Corollary}

\newtheorem{def/prop}[theorem]{Definition/Proposition}

\theoremstyle{definition}

\newtheorem{remark}[theorem]{Remark}
\newtheorem*{ack}{Acknowledgements}
\newtheorem{example}{Example}[section]

\DeclareSymbolFont{bbold}{U}{bbold}{m}{n}
\DeclareSymbolFontAlphabet{\mathbbold}{bbold}

\def\BOne{\mathchoice{\scalebox{1.16}{$\displaystyle\mathbbold 1$}}{\scalebox{1.16}{$\textstyle\mathbbold 1$}}{\scalebox{1.16}{$\scriptstyle\mathbbold 1$}}{\scalebox{1.16}{$\scriptscriptstyle\mathbbold 1$}}}

\def\fract#1#2{\raise4pt\hbox{$ #1 \atop #2 $}}
\def\decdnar#1{\phantom{\hbox{$\scriptstyle{#1}$}}
\left\downarrow\vbox{\vskip15pt\hbox{$\scriptstyle{#1}$}}\right.}

\def\lcm{{\rm lcm}}

\def\bbc{{\mathbb C}}

\def\bbp{{\mathbb P}}
\def\bbq{{\mathbb Q}}
\def\bbr{{\mathbb R}}

\def\bbz{{\mathbb Z}}

\def\gra{\alpha}
\def\grb{\beta}

\def\grg{\gamma}

\def\gro{\omega}

\def\grt{\tau}

\def\grD{\Delta}

\def\grL{\Lambda}

\def\bfa{{\bf a}}

\def\bfl{{\bf l}}
\def\bfm{{\bf m}}
\def\bfn{{\bf n}}

\def\bfr{{\bf r}}

\def\bfv{{\bf v}}
\def\bfw{{\bf w}}

\def\cald{{\mathcal D}}

\def\calf{{\mathcal F}}

\def\cali{{\mathcal I}}

\def\calk{{\mathcal K}}
\def\call{{\mathcal L}}

\def\calo{{\mathcal O}}

\def\cals{{\mathcal S}}

\def\calz{{\mathcal Z}}

\def\gs{{\mathfrak s}}
\def\gt{{\mathfrak t}}

\def\gz{{\mathfrak z}}

\def\lcm{{\rm lcm}}

\def\lra{\longrightarrow}

\def\la#1{\hbox to #1pc{\leftarrowfill}}
\def\ra#1{\hbox to #1pc{\rightarrowfill}}

\def\<{\langle}
\def\>{\rangle}

\begin{document}

\title[CSC Sasaki Metrics and Admissible Construction]{Constant Scalar Curvature Sasaki Metrics and Projective Bundles}

\author[Charles Boyer]{Charles P. Boyer}
\address{Charles P. Boyer, Department of Mathematics and Statistics,
University of New Mexico, Albuquerque, NM 87131.}
\email{cboyer@unm.edu} 
\author[Christina T{\o}nnesen-Friedman]{Christina W. T{\o}nnesen-Friedman}
\address{Christina W. T{\o}nnesen-Friedman, Department of Mathematics, Union
College, Schenectady, New York 12308, USA }
\email{tonnesec@union.edu}
\thanks{The authors were partially supported by grants from the
Simons Foundation, CPB by (\#519432), and CWT-F by (\#422410)}
\date{\today}

\begin{abstract}
In this paper we consider the Boothby-Wang construction over twist 1 stage 3 Bott orbifolds given in terms of the log pair $(S_\bfn,\grD_\bfm)$. We give explicit constant scalar curvature (CSC) Sasaki metrics either directly from CSC K\"ahler orbifold metrics or by using the weighted extremal approach of Apostolov and Calderbank. The Sasaki 7-manifolds (orbifolds) are finitely covered by compact simply connected manifolds (orbifolds) with the rational homology of the 2-fold connected sum of $S^2\times S^5$. 
\end{abstract}

\maketitle
\vspace{-7mm}

\tableofcontents

\section*{Introduction}\label{intro}
Arguably the foremost problem in Sasaki geometry is that of determining which isotopy classes of Sasakian structures admit Sasaki metrics of constant scalar curvature (CSCS). The Sasaki version of the Yau-Tian-Donaldson conjecture says roughly that CSCS metrics correspond to the affine variety satisfying some type of K-stability requirement. Collins and Sz\'ekelyhidi \cite{CoSz12}  proved that a CSC Sasaki metric implies the K-semistablility of corresponding affine K\"ahler variety. Recently, following the important work in the K\"ahler category by Chen and Cheng \cite{ChChI21,ChChII21}, He and Li \cite{He18,HeLi21} have related the existence of CSC Sasaki metrics to the reduced properness of the Mabuchi K-energy.

Complementary to these existence results, it is important to explore explicit examples. Of particular interest to us are those Sasaki manifolds (orbifolds) that admit a transverse admissible construction \cite{ACGT08}. A special case of such Sasaki orbifolds consist of principle $S^1$ orbibundles over Bott manifolds (orbifolds) \cite{GrKa94,BoCaTo17} to be described in the first Section. In this paper we discuss two fairly explicit methods for constructing constant scalar curvature Sasaki metrics on a class of 7-manifolds having the rational cohomology of the 2-fold connected sum of $S^2\times S^5$. 
Specifically, these $7$-manifolds, $M^7_{\bfn, \bfm}$, arise by a Boothby-Wang construction over a polarized admissible orbifold, $S_{\bfn,\bfm}$, given as a log pair $(S_\bfn, \Delta_\bfm)$, where
$S_\bfn=\bbp\bigl(\BOne\oplus \calo(n_1,n_2)\bigr)\lra \bbc\bbp^1\times\bbc\bbp^1$ and $\grD_\bfm=\Bigl(1-\frac{1}{m_0}\Bigr)e_0+\Bigl(1-\frac{1}{m_\infty}\Bigr)e_\infty$, with
$e_0$ and $e_\infty$ are the infinity sections of the bundle $S_\bfn$.
The first method is the standard Boothby-Wang construction on twist one stage 3 Bott orbifolds (KS orbifolds). The second method applies the weighted extremal approach of Apostolov and Calderbank \cite{ApCa18}. Both methods employ the admissible construction at the K\"ahler level (see \cite{ACGT08}) either directly or indirectly. Specifically, in Section \ref{adsect}, with the main work horses being Proposition \ref{CRtwistCSC} and Proposition \ref{orbtocsc} (the latter with limitations as mentioned in Remark \ref{4.14caution}), we prove Theorem \ref{CSCSasakiExistence} which together with Remark \ref{cscremark} give

\noindent {\bf Main Theorem}. {\it Let $M^7_{\bfn, \bfm}$ be a Boothby-Wang constructed manifold with maximal symmetry over a polarized KS orbifold. Then there always exists a positive constant scalar curvature Sasaki metric in the Sasaki cone.}

In terms of K-stability we have the following corollary of the Main Theorem and Corollary 1.1 of \cite{CoSz12}

\noindent {\bf Corollary}. {\it Let $M^7_{\bfn, \bfm}$ be a Boothby-Wang constructed Sasaki manifold over the KS orbifold $(S_\bfn,\Delta_\bfm)$. Then there exists a Reeb vector field $\xi$ in the Sasaki  cone such that the corresponding polarized affine variety $(Y,\xi)$
is K-semistable, equivalently the Sasakian structure $(\xi,\eta,\Phi,g)$ is K-semistable.}

In the Gorenstein case, a much stronger statement is known, namely the Sasaki version of the famous Chen-Donaldson-Sun \cite{ChDoSu15I,ChDoSu15II,ChDoSu15III} result by Collins and Sz\'ekelyhidi \cite{CoSz15}. In particular, the recent classification of all smooth Fano threefolds that admit K\"ahler-Einstein metrics \cite{ACCFKMSSV} gives a classification of all regular Sasaki-Einstein manifolds over smooth KS threefolds.

It is interesting to compare the Sasaki-Einstein solutions obtained from the Boothby-Wang construction from the orbifold K\"ahler-Einstein solutions of Section \ref{KEsection} with those described in \cite{BoTo18c}. Both types live on simply connected 7-manifolds with the rational cohomology of the 2-fold connected sum of $S^2\times S^5$ with torsion in $H^4$. However, they are of a somewhat different nature. Those in \cite{BoTo18c} have holomorphic twist 2, whereas, those in this paper have holomorphic twist 1. Assuming they have the same torsion group one may ask whether they are homotopy equivalent, have the same diffeomorphism type, and/or belong to inequivalent CR structures, contact structures, etc.

\begin{ack}
The authors of this paper have benefitted from conversations with Vestislav Apostolov, David Calderbank, Hongnian Huang, Eveline Legendre, and Carlos Prieto. 
\end{ack}

\section{The Projective Bundles}
We view our projective bundle as a special type of stage 3 Bott manifold, namely the Bott tower $M_3(0,n_1,n_2)$ where $n_1,n_2\in \bbz$. We recall the definition of a \emph{stage $k$ Bott manifold} of the \emph{Bott tower of
  height $n$}:
\[
M_n \xrightarrow{\pi_n}M_{n-1} \xrightarrow{\pi_{n-1}} \cdots\rightarrow M_2
\xrightarrow{\pi_2} M_1 =\bbc\bbp^1 \xrightarrow{\pi_1} \{pt\}
\]
where $M_k$ is defined inductively as the total space of the projective bundle 
$$\bbp(\BOne\oplus L_k)\xrightarrow{\pi_k}M_{k-1}$$ 
with fiber $\bbc\bbp^1$, and some holomorphic line bundle $L_k$ on $M_{k-1}$.We refer to \cite{GrKa94,BoCaTo17} for details. In this paper we treat twist 1 stage 3 Bott manifolds $M_3(0,n_1,n_2)$ and orbifolds. Explicitly these are ruled manifolds of the form 
\begin{equation}\label{Sneqn}
S_\bfn=\bbp\bigl(\BOne\oplus \calo(n_1,n_2)\bigr)\lra \bbc\bbp^1\times\bbc\bbp^1
\end{equation} 
with $n_1,n_2\in\bbz^*$. To this we add an orbifold structure given in terms of the log pair $(S_\bfn,\grD_\bfm)$  where 
\begin{equation}\label{grDeqn}
\grD_\bfm=\Bigl(1-\frac{1}{m_0}\Bigr)e_0+\Bigl(1-\frac{1}{m_\infty}\Bigr)e_\infty
\end{equation}
and $e_0,e_\infty$ are defined by the zero and infinity sections of the bundle. 
A motivation of our study comes from the 1986 paper of Koiso and Sakane \cite{KoSa86} where they study K\"ahler-Einstein metrics on projectivizations of the form $\bbp(\BOne\oplus L_1\otimes_{ext} L_2)$ over $N\times N$ where $N$ has a K\"ahler-Einstein metric. For this reason we refer to our log pairs $(S_\bfn,\grD_\bfm)$ as {\it KS orbifolds}. The underlying complex manifolds are smooth toric varieties whose K\"ahler geometry was studied in \cite{BoCaTo17}. 
The positive integers $m_0,m_\infty$ are called {\it ramification indices}. For each choice of orbifold K\"ahler form $\omega_{\bfn,\bfm}$, defining an integer K\"ahler class on $(S_\bfn,\Delta_\bfm)$ we have a principal $S^1$ orbibundle
$$S^1\lra M^7\lra (S_\bfn,\grD_\bfm)$$ 
with an induced Sasakian structure $\cals_{\bfn,\bfm}$; moreover, when the class $[\gro_{\bfn,\bfm}]$ is primitive in $H^2_{orb}((S_\bfn,\grD_\bfm),\bbz)$, $M^7$ is simply connected. While there are only 2 equivalences classes of Fano KS manifolds, it is clear from Proposition \ref{orbFano} below that log Fano KS orbifolds are abundant.

In \cite{BoTo20a} we described a functor from the category of $S^3_\bfw$ Sasaki joins to the category of Bott orbifolds. However, as noted there, this functor is not surjective on objects. One purpose of this paper is to study this lack of surjectivity in the particular case of twist 1 stage 3 Bott orbifolds. 
A key notion of Sasaki joins is that of {\it cone decomposability} which can be thought of as the Sasaki version of de Rham decomposability in the K\"ahler case \cite{BHLT16}. Cone  decomposability is an invariant of the underlying Sasaki CR structure. We shall see how our KS orbifolds give rise to both cone decomposable Sasakian structures as well as Sasakian structures that are likely cone indecomposable. See Remark \ref{dec-indec} below.

Note that the ordinary Bott manifold $M_3(0,n_1,n_2)$ corresponds to the trivial orbifold $(S_\bfn,\emptyset)$. We denote by $KS^{orb}$ the set of  manifolds with these orbifold structures, that is,
$$KS^{orb}=\{(S_\bfn,\grD_\bfm)~|~n_1,n_2\in\bbz^*,~m_0,m_\infty\in\bbz^+\}.$$
Such orbifolds are the objects of a groupoid whose morphisms are biholomorphisms that intertwine the orbifold structures. Actually we are intererested in {\it polarized} KS orbifolds that are polarized by a (primitive) orbifold K\"ahler class $[\gro_{\bfn,\bfm}]\in H^2_{orb}(S_\bfn,\bbz)$. These form the objects of our groupoid $\calk\cals$ whose morphisms are biholomorphisms that intertwine the orbifold structures and intertwine their orbifold K\"ahler classes. We remark that since the objects of $\calk\cals$ are themselves orbifolds, we are working with a 2-category. The action of the Coexter group $\mathrm{Sym}_2\ltimes\bbz_2^3$ on $KS_{orb}$ is generated by the transposition $(S_\bfn,\grD_\bfm)\mapsto (S_{n_2,n_1},\grD_\bfm)$ and the fiber inversion sending $(S_\bfn,\grD_\bfm)\mapsto (S_{-\bfn},\grD_{m_\infty,m_0})$. Note the interchange of ramification indices $m_0,m_\infty$ induced by the fiber inversion map.

Of special interest is the full subgroupoid $\calk\cals^{mon}$ of monotone KS orbifolds which we now describe. 
Consider the orbifold canonical divisor or dually the orbifold first Chern class. In particular, when is the orbifold $(S_\bfn,\grD_\bfm)$ log Fano, equivalently when does $c_1^{orb}(S_\bfn,\grD_\bfm)$ lie in the K\"ahler cone?
From Section 1.3 of \cite{BoTo18c} we have
\begin{eqnarray}\label{c1orb2}
c_1^{orb}(S_\bfn,\grD_\bfm)   &=& \Bigl(2-\frac{n_1}{m_\infty}\Bigr)y_1 +\Bigl(2-\frac{n_2}{m_\infty}\Bigr)y_2+\Bigl(\frac{1}{m_0}+\frac{1}{m_\infty}\Bigr)y_3 \notag \\
                                             &=& \Bigl(2+\frac{n_1}{m_0}\Bigr)x_1 +\Bigl(2+\frac{n_2}{m_0}\Bigr)x_2+\Bigl(\frac{1}{m_0}+\frac{1}{m_\infty}\Bigr)x_3. \end{eqnarray}
with respect to the invariant bases $\{x_1,x_2,x_3\}$ and $\{y_1,y_2,y_3\}$, respectively.
For $i=1,2$ here $x_i=y_i$ is the class of the Fubini-Study metric on the ith factor of the product $\bbc\bbp^1\times \bbc\bbp^1$ pulled back to $S_\bfn$, while $x_3 (y_3)$ is the Poincar\'e dual of $e_\infty (e_0)$, respectively. From this we arrive at a special case of Lemma 1.2 of \cite{BoTo18c}
\begin{proposition}\label{orbFano}
The orbifold $(S_\bfn,\grD_\bfm)$ in $KS^{orb}$ is log Fano if and only if the inequalities
$$\frac{n_1}{m_\infty}<2,\quad \frac{n_2}{m_\infty}<2,\quad -\frac{n_1}{m_0}<2,\quad -\frac{n_2}{m_0}<2$$
hold.
\end{proposition} 

We obtain a primitive K\"ahler class $\frac{c_1^{orb}(S_\bfn,\grD_\bfm)}{\cali_{\bfn,\bfm}}$ in $H^2_{orb}((S_\bfn,\grD_\bfm),\bbz)$ with  {\it Fano index} given by
\begin{equation}\label{index}
\cali_{\bfn,\bfm}=\gcd(2mv_0v_\infty- n_1v_0,2mv_0v_\infty- n_2v_0,v_0+v_\infty)
\end{equation}
where $\bfm=(m_0,m_\infty)=m(v_0,v_\infty)=m\bfv$ and $m=\gcd(m_0,m_\infty)$. Here we have used the identity $x_3=y_3-n_1x_1-n_2x_2$. The orbifold $(S_\bfn,\grD_\bfm)$ is log Fano if and only if 
$$c_1^{orb}(S_\bfn,\grD_\bfm)= \frac{1}{mv_0v_\infty}\Bigl[\Bigl(2mv_0v_\infty- n_1v_0\Bigr)y_1 +\Bigl(2mv_0v_\infty- n_2v_0\Bigr)y_2+\Bigl(v_0+v_\infty\Bigr)y_3\Bigr] \label{c1orb3}$$
 is positive.

The Poincar\'e dual to $c_1^{orb}(S_\bfn,\grD_\bfm)$ is the orbifold anti-canonical divisor which is an ample $\bbq$-divisor on $S_\bfn$ when $c_1^{orb}(S_\bfn,\grD_\bfm)$ is positive. The primitive class $\frac{c_1^{orb}(S_\bfn,\grD_\bfm)}{\cali_{\bfn,\bfm}}$ can be represented by an orbifold K\"ahler form $\gro_{\bfn,\bfm}$ on $S_\bfn$, or equivalently an orbifold K\"ahler metric $g_{\bfn,\bfm}$.

\section{The Orbifold Boothby-Wang Construction}\label{sassec}
The well known Boothby-Wang Theorem \cite{BoWa} that associates regular contact structures to the total space of $S^1$ bundles over symplectic manifolds generalizes easily to the orbifold category. See for example Theorem 7.1.3 of \cite{BG05}. Here as in this theorem we are interested in Sasakian structures over K\"ahlerian structures. Given an integral K\"ahler class $[\gro_{\bfn,\bfm}]\in H^{1,1}_{orb}((S_\bfn,\grD_\bfm),\bbz)$ we shall always choose a K\"ahler form $\gro_{\bfn,\bfm}$ with maximal symmetry. In this case the total space $M$ of the corresponding principal $S^1$ orbibundle has a maximal family $\gt^+$ of Sasakian structures, called the {\it Sasaki cone}. Moreover, choosing a connection 1-form $\eta$ on $M$ such that $d\eta=\gro_{\bfn,\bfm}$ determines a natural Sasakian structure $\cals_{\bfn,\bfm}=(\xi,\eta,\Phi,g)$ in $\gt^+$. We refer to this construction as the {\it orbifold Boothby-Wang construction}. Generally, unlike $(S_\bfn,\grD_\bfm)$, the total space $M$ has an orbifold structure whose underlying topological space is singular with finite cyclic singularities. The general problem of finding conditions when $M$ is smooth is quite subtle \cite{Qui71,Kol04,KeLa21}: however, Kegel and Lange show that  $M$ is a smooth manifold if and only if $H^r_{orb}(M,\bbz)=0$ for all $r>2n+1$. More precisely  $M$ is smooth if and only if multiplication by the Euler class $ H^r_{orb}(\calz,\bbz)\fract{\cup_e}{\ra{3.0}} H^{r+2}_{orb}(\calz,\bbz)$
in the Gysin sequence of the orbibundle $M\ra{2.0}\calz$
is an isomorphism for all $r>2n+1$. As we shall see in Section \ref{joinsect} precise conditions can be obtained in certain special cases (see also \cite{BG05}). 

The question arises as to when there exists a constant scalar curvature Sasaki metric in $\gt^+$ and how many. In the Gorenstein case this was answered by Futaki, Ono, and Wang \cite{FOW06}, but the general CSC case remains open. Some partial results have been obtained by Legendre \cite{Leg10}, and by the authors \cite{BoTo14a,BoTo19a} and references therein.

\begin{remark}
In categorical language we work with the groupoid $\cals\calk\cals$ whose objects $\{M^7_{\bfn,\bfm}\}$ are simply connected quasiregular Sasaki 7-orbifolds whose projective algebraic base is a Koiso-Sakane orbifold, and whose morphisms are orbifold biholomorphisms. Here the objects are classes of Sasakian or Koiso-Sakane structures, thus again, they are 2-categories. We are mainly interested in certain subgroupoids, namely, the full subgroupoid $\cals\calk\cals^{F}$ of positive Sasakian structures whose base Koiso-Sakane orbifold is log Fano. 
Actually, we consider the full subgroupoid $\cals\calk\cals^{MF}$ whose K\"ahler form $\gro_{\bfn,\bfm}$ lies in $\frac{c_1^{orb}(S_\bfn,\grD_\bfm)}{\cali_{\bfn,\bfm}}$, that is the Koiso-Sakane base orbifold is $(S_\bfn,\grD_\bfm)$ is monotone log Fano.   
We also define the subgroupoids $\cals\calk\cals_\bfm,\cals\calk\cals^{F}_\bfm,\cals\calk\cals^{MF}_\bfm$ consisting of the corresponding objects with fixed $\bfm$. The objects of the groupoid $\cals\calk\cals_\bfm$ are Sasaki classes on the 7-manifolds $M^7_\bfm$ whose cohomology we describe in the next section. The morphisms are those induced by diffeomorphisms that intertwine the corresponding Sasakian structures. 
\end{remark}

\section{The Topology of the Orbifolds}
In this section we describe both the orbifold topology of $(S_\bfn,\grD_\bfm)$ and topology of the total space of the principal $S^1$ orbibundles over the log Fano Koiso-Sakane orbifolds $(S_\bfn,\grD_\bfm)$. First we recall the cohomology ring of the complex manifold $S_\bfn$ which is well known and can be found in \cite{BoCaTo17} and references therein.
\begin{equation}\label{Sncohring}
H^*(S_\bfn,\bbz)=\bbz[y_1,y_2,y_3]/\bigl(y_1^2,y_2^2,y_3(-n_1y_1-n_2y_2+y_3)\bigr)
\end{equation}
We remark that the addition of orbifold structures on the invariant divisors does not effect the cohomology $H^*(S_\bfn,\bbz)$; however, as we shall see it strongly effects the cohomology of the Sasakian 7-manifolds.

\subsection{The Orbifold Cohomology Groups}
Next we compute the orbifold cohomology groups.

\begin{lemma}\label{orbcoh}
$$H^{r}_{orb}((S_\bfn,\grD_\bfm),\bbz)=\begin{cases} \bbz &~~\text{if $r=0$} \\
                                     \bbz^3&  ~~\text{if $r=2$} \\
                                     \bbz^3 \oplus \bbz_{m^0}^2\oplus \bbz_{m^\infty}^2&  ~~\text{if $r=4$} \\
                               \bbz\oplus  \bbz_{m^0}^3\oplus \bbz_{m^\infty}^3 & ~~\text{if $r=6$} \\
                                \bbz_{m^0}^3\oplus \bbz_{m^\infty}^3 & ~~\text{if $r=8,10,\cdots$} \\
                                  0 & ~~\text{if $r$ is odd.}
                                  \end{cases}$$
\end{lemma}

\begin{proof}
Using Lemma 4.3.7 of \cite{BG05} we compute the Leray spectral sequence of the classifying map 
$$p:\mathsf{B}(S_\bfn,\grD_\bfm)\ra{2.5} ~S_\bfn.$$ 
The Leray sheaf of $p$ is the derived functor sheaf $R^sp_*\bbz$, that is, the sheaf associated to the presheaf $U\mapsto H^s(p^{-1}(U),\bbz)$. For $s>0$ the stalks of $R^sp_*\bbz$ at points of $U$ vanish if $U$ lies in the regular locus of $(S_\bfn,\grD_\bfm)$ which is the complement of the union of the zero $e_0$ and infinity $e_\infty$ sections of the natural projection $S_\bfn\ra{1.6} \bbc\bbp^1\times \bbc\bbp^1$.  At points of $e_0$ and  $e_\infty$ the fibers of $p$ are the Eilenberg-MacLane spaces $K(\bbz_{m^0},1)$ and $K(\bbz_{m^\infty},1)$, respectively. So at points of $e_0(e_\infty)$ the stalks are the group cohomology $H^s(\bbz_{m^0},\bbz)\bigl(H^s(\bbz_{m^\infty},\bbz)\bigr)$. This is $\bbz$ for $s=0$ and $\bbz_{m^0}(\bbz_{m^\infty})$ at points of $e_0(e_\infty)$ when $s>0$ is even; it vanishes when $s$ is odd. The $E_2$ term of the Leray spectral sequence of the map $p$ is 
$$E_2^{r,s}=H^r(S_\bfn,R^sp_*\bbz)$$
and by Leray's theorem this converges to the orbifold cohomology  $H^{r+s}_{orb}((S_\bfn,\grD_\bfm),\bbz)$. Now $E_2^{r,0}=H^r(S_\bfn,\bbz)$ and $E_2^{r,s}=0$ for $r$ or $s$ odd. For $r=0$ the only continuous section of $R^sp_*\bbz$ is the 0 section which implies that $E_2^{0,s}=0$ for all $s$. Now we have $E_2^{2r,2s}=0$ for $r>2$ and  
\begin{eqnarray*}
E_2^{2,2s}&=&H^2(S_\bfn,R^{2s}p)=H^2(e_0,\bbz_{m^0})\oplus  H^2(e_\infty,\bbz_{m^\infty})=\bbz_{m^0}^2\oplus \bbz_{m^\infty}^2, \\
E_2^{4,2s}&=&H^4(S_\bfn,R^{2s}p)=H^4(e_0,\bbz_{m^0})\oplus  H^4(e_\infty,\bbz_{m^\infty})=\bbz_{m^0}\oplus \bbz_{m^\infty}.
\end{eqnarray*}
One easily sees that this spectral sequence collapses whose limit is the orbifold cohomology $H^{r}_{orb}((S_\bfn,\grD_\bfm),\bbz)$ which implies the result. 
\end{proof}

In addition Lemma \ref{orbcoh} implies

\begin{lemma}\label{pi1orblem}
$\pi_1^{orb}(S_\bfn,\grD_\bfm)=\BOne.$
\end{lemma}

\begin{proof}
From the homotopy sequence of the orbibundle 
$$\bbc\bbp^1[\bfv]/\bbz_m\lra (S_\bfn,\grD_\bfm)\lra \bbc\bbp^1\times \bbc\bbp^1$$
one easily sees that $\pi_1^{orb}(S_\bfn,\grD_\bfm)$ is Abelian. But this implies 
$$\pi_1^{orb}(S_\bfn,\grD_\bfm)\approx H_1^{orb}(S_\bfn,\grD_\bfm),\bbz),$$
and this vanishes by Lemma \ref{orbcoh} and universal coefficients.
\end{proof}

\subsection{The Cohomology of the Branched Covers}
We consider the total space $M^7_{\bfn,\bfm}$ of an $S^1$ orbibundle over a KS orbifold $(S_\bfn,\grD_\bfm)$ with an orbifold K\"ahler form $\gro_{\bfn,\bfm}$. Note that in our case $S_\bfn$ is a smooth projective variety so $M^7_{\bfn,\bfm}$ is a Seifert bundle as defined in the foundational paper \cite{OrWe75} of Orlik and Wagreich (see also the more recent unpublished description by Koll\'ar \cite{Kol04c}). A fundamental result of Orlik and Wagreich says that every $\bbc^*$ Seifert bundle is a branched cover over a principal $\bbc^*$ bundle. In our case this branching occurs precisely along the orbifold divisor $\grD_\bfm$. The goal is to understand the integral cohomology of $M^7_{\bfn,\bfm}$. However, before treating the orbifold case, we need to consider the regular case.

\subsection{The Regular Case.}
For simplicity we treat only the case where the K\"ahler class is primitive, i.e.
\begin{equation}\label{primKahcl}
[\gro_\bfn]= c_1y_1+c_2y_2+c_3 y_3 ~\text{with $c_i\in\bbz^+$}, \qquad \gcd(c_1,c_2,c_3)=1.
\end{equation}
From Section 3.3 of \cite{BoCaTo17} this is a K\"ahler class for all positive $c_i$ if $n_1n_2>0$, whereas, if $n_1n_2<0$, it is a K\"ahler class for $c_1>0,~c_2>-n_2c_3,~c_3>0$.

One sees from the long exact homotopy sequence of the fibration 
\begin{equation}\label{regfib}
S^1\lra M^7_\bfn\lra S_\bfn
\end{equation}
that

\begin{lemma}\label{reglemtop}
Let $M^7_\bfn$ denote the total space of the principal $S^1$ bundle over a KS manifold $S_\bfn$ defined by a primitive K\"ahler class $[\gro_\bfn]\in H^2(S_\bfn,\bbz)$. Then $M^7_\bfn$ has the rational cohomology of the 2-fold connected sum $2\# (S^2\times S^5)$, and we have
\begin{enumerate}
\item $\pi_1(M^7_\bfn)=\BOne$,
\item $\pi_2(M^7_\bfn)= \bbz^2$.
\end{enumerate}
\end{lemma}
Thus, by the Hurewicz Theorem we also have $H_2(M^7_\bfn,\bbz)=\bbz^2$ which implies that $H^3(M^7_\bfn,\bbz)_{tor}=0$. One can also see from the Leray-Serre spectral sequence of the fibration \eqref{regfib} that $M^7_\bfn$ has the rational cohomology of the 2-fold connected sum $2\# (S^2\times S^5)$, and we also have $H^3(M^7_\bfn,\bbz)=0$. We now show that all such simply connected spaces have nonzero torsion in $H^4$. There will be zero torsion in $H^4$ if and only if the differential $d_2:E^{2,1}_2\lra E^{4,0}_2$ is invertible over $\bbz$. This differential is represented by the matrix 
\begin{equation}\label{Cmatrix}
C=
\begin{pmatrix}
c_2 & c_3 & 0 \\
c_1 & 0 & c_3 \\
0 & c_1+n_1c_3 & c_2+n_2c_3
\end{pmatrix}
\end{equation}
with respect to the basis $\{\gra\otimes y_1,\gra\otimes y_2,\gra\otimes y_3\}$ of $E^{2,1}_2$ and the basis $\{y_1y_2,y_1y_3,y_2y_3\}$ of $E^{4,0}_2$. But then we have
\begin{equation}\label{detC}
-\det C=c_3[c_2(c_1+n_1c_3)+c_1(c_2+n_2c_3)]
\end{equation}
which implies that the order of $H^4(M^7_\bfn,\bbz)$ is $c_3[c_2(c_1+n_1c_3)+c_1(c_2+n_2c_3)]$. This is always greater than $1$, so $H^4(M^7_\bfn,\bbz)_{tor}\neq 0$. Summarizing we have

\begin{theorem}\label{regs1bun}
Let $M^7_\bfn$ denote the total space of the principal $S^1$ bundle over a KS manifold $S_\bfn$ defined by a primitive K\"ahler class $[\gro_\bfn]=c_1x_1+c_2x_2+c_3x_3.$ Then 
$$H^{r}(M^7_{\bfn},\bbz)=\begin{cases} \bbz &~~\text{if $r=0,7$} \\
                                     \bbz^2&  ~~\text{if $r=2,5$} \\
                                     G_{reg} &  ~~\text{if $r=4$} \\
                                     0 & ~~\text{otherwise.}
                                  \end{cases}$$
where $G_{reg}$ is an Abelian group of order $c_3[c_2(c_1+n_1c_3)+c_1(c_2+n_2c_3)]$. Moreover, $G_{reg}$ is never the identity.
\end{theorem}

\begin{example}\label{stanprod}
The simplist case is the standard product $\bbc\bbp^1\times\bbc\bbp^1\times\bbc\bbp^1$ with K\"ahler class 
$$[\gro]=c_1x_1+c_2x_2+c_3x_3,\qquad \gcd(c_1,c_2,c_3)=1$$
and $n_1=n_2=0$ which gives $|G_{reg}|=2c_1c_2c_3$. Taking $c_1=c_2=c_3=1$ gives the homogeneous space 
$$M^7=\bigl(SU(2)\times SU(2)\times SU(2)\bigr)/U(1)\times U(1)$$
with its Sasaki-Einstein metric (see for example \cite{BFGK91}). Here $G_{reg}=\bbz_2$.

\end{example}

\begin{example}Monotone Fano case.
Up to equivalence there are two monotone Fano cases polarized by the first Chern class, namely, $\bfn=(1,-1)$ and $\bfn=(1,1)$.
From Equations \eqref{c1orb2} and \eqref{primKahcl} we have $c_1=1,c_2=3,c_3=2$ when $\bfn=(1,-1)$, and $c_1=1,c_2=1,c_3=2$ when $\bfn=(1,1)$.
These give $|G_{reg}|=20$ and $|G_{reg}|=12$, respectively. Note that $M^7_{1,-1}$ is conjecturely\footnote{See Remark \ref{dec-indec} below}  cone indecomposable; whereas, $M^7_{1,1}$ is cone decomposible. The latter is described by Theorem 2.5 in \cite{BoTo20a} with $\bfl_1=(1,1), \bfw_1=(1,1), \bfl_2=(1,2)$, and
$\bfw_2=(3,1)$. 
               
\end{example}

\subsection{The General Case; Branched Covers}
Returning to the orbifold case, we consider the Leray spectral sequence of the quotient map $M^7_{\bfn,\bfm}\fract{\pi}{\lra}(S_\bfn,\grD_\bfm)$ viewed as an $S^1$-Seifert bundle over $(S_\bfn,\grD_\bfm)$. Since the fibers are $S^1$ we only have the $R^1\pi_*\bbz_M$ direct image sheaf, the $E_2$ term of the Leray spectral sequence
\begin{equation}\label{E2Leray}
E^{p,q}_2=H^{p}(S_\bfn,R^q\pi_*\bbz)
\end{equation}
satisfies $E^{p,q}_2=0$ when $q\neq 0,1$, and this converges to $H^{p+q}(M^7_{\bfn,\bfm},\bbz)$. So the only differentials are those induced by  
\begin{equation}\label{d2map}
d_2:E^{0,1}_2=H^0(S_\bfn, R^1\pi_*\bbz_M)\lra E^{2,0}_2=H^2(S_\bfn,\bbz)\approx \bbz^3.
\end{equation} 
We note that the sheaf injection 
\begin{equation}\label{sheafmap1}
R^1\pi_*\bbz_M\lra \bbz_{S_\bfn}
\end{equation}
is multiplication by $1$ on the regular locus, multiplication by $m_0$ on $e_0$ and by $m_\infty$ on $e_\infty$. So the image of the map \eqref{d2map} is the order $\mu=\lcm(m_0,m_\infty)$ times a primitive orbifold K\"ahler class $[\gro_{\bfn,\bfm}]\in H^2_{orb}((S_\bfn,\grD_\bfm),\bbz)$. To compute the cohomology of such spaces we consider the $\bbc^*$ Seifert bundle over $(S_\bfn,\grD_\bfm)$ which is the cone $C(M^7_{\bfn,\bfm})=Y_{\bfn,\bfm}=M^7_{\bfn,\bfm}\times\bbr^+ $  so that $Y_{\bfn,\bfm}|_{r=1}=M^7_{\bfn,\bfm}$. Now $Y_{\bfn,\bfm}$ is a branched cover of $Y_{\bfn,\bfm}/\bbz_\mu$ with branching locus $\grD_\bfm$. This follows from \cite{Kol04c} by identifying the orbifold class $[\gro_{\bfn,\bfm}]$ with the rational first Chern class $c_1(Y/X)$ in \cite{Kol04c} with $X=S_\bfn$. Now generally the total space $M^7_{\bfn,\bfm}$ is a Sasaki orbifold whose underlying topological space is a compactly generated Hausdorff space which allows us to apply the theory of ramified covers \cite{Smi83,AgPr06}. The space $M^7_{\bfn,\bfm}$ is simply connected, and a $\mu$-fold cover of the total space of the ordinary $S^1$ bundle defined by the primitive integral class $\mu[\gro_{\bfn,\bfm}]$. Thus, we have the commutative diagram of Seifert orbibundles:
\begin{equation}\label{commdiag2}
\begin{matrix}
S^1 &\ra{2.0} &M^7_{\bfn,\bfm}& \fract{f}{\ra{2.0}} &(S_\bfn,\grD_\bfm) \\
\decdnar{}&&\decdnar{\pi}&&\decdnar{}\\
S^1/\bbz_\mu &\ra{2.0} &M^7_{\bfn,\bfm}/\bbz_\mu& \fract{f_\mu}{\ra{2.0}} &(S_\bfn,\emptyset).
\end{matrix}
\end{equation}
Note that since $M^7_{\bfn,\bfm}/\bbz_\mu$ is the total space of a principal $S^1$ bundle over the simply connected smooth projective algebraic variety $S_\bfn$ defined by a primitive integral class $\mu[\gro_{\bfn,\bfm}]$, it is a simply connected smooth 7-manifold. Moreover, the covering map $\pi$ induces a map of the corresponding Leray spectral sequences, and since the fiber is an $S^1$ the only nonzero higher direct image sheaf is $R^1\pi_*\bbz$. This gives the natural isomorphism $R^1\pi_*\bbq_M\approx \bbq_{S_\bfn}$ which then implies

\begin{lemma}\label{Qiso}
There is an isomorphism $H^*(M^7_{\bfn,\bfm},\bbq)\approx H^*(M^7_{\bfn,\bfm}/\bbz_\mu,\bbq)$.
\end{lemma}

However, we would also like information about integral cohomology groups. For this we consider the {\bf transfer homomorphism} for ramified covers \cite{Smi83}. Following \cite{AgPr06} we apply this to our branched cover $\pi: M^7_{\bfn,\bfm}\ra{2.0} M^7_{\bfn,\bfm}/\bbz_\mu$ . Theorem 5.4 of \cite{AgPr06} says that the transfer homomorphism 
$$\grt:H^*(M^7_{\bfn,\bfm},\bbz)\ra{2.0} H^*(M^7_{\bfn,\bfm}/\bbz_\mu,\bbz)$$
induces multiplication by $\mu$ in $H^*(M^7_{\bfn,\bfm}/\bbz_\mu,\bbz)$. We set $G_{reg}=H^4(M^7_{\bfn,\bfm}/\bbz_\mu,\bbz)$ and apply this to $*=4$. 
If $\gcd(|G|,\mu)=1$ we see that 
$$\grt\circ \pi^*:G_{reg}=H^4(M^7_{\bfn,\bfm}/\bbz_\mu,\bbz)\ra{2.0} G_{reg}=H^4(M^7_{\bfn,\bfm}/\bbz_\mu,\bbz)$$ 
is an isomorphism which implies that $\pi^*:H^4(M^7_{\bfn,\bfm}/\bbz_\mu,\bbz)\ra{2.0} H^4(M^7_{\bfn,\bfm},\bbz)$ is injective. Thus, $\pi^*(G_{reg})$ is a subgroup of $H^4(M^7_{\bfn,\bfm},\bbz)$ in this case. When  $\gcd(|G_{reg}|,\mu)\neq 1$ the homomorphism $\grt\circ \pi^*$ has a non-trivial kernel. This kernel consists of all the prime factors $\bbz_{p_i^{r_1}}\times\cdots \times \bbz_{p_k^{r_k}}$ of $G_{reg}$ such that $p_i$ is a prime in $\mu$. Denote this group by $G_{reg}^\mu$ In this case $H^4(M^7_{\bfn,\bfm},\bbz)$ contains the subgroup of $G_{reg}$ containing only those prime factors whose primes $p$ are not in $\mu$, that is the factor group $G_{reg}/G_{reg}^\mu$. Summarizing we have 

\begin{lemma}\label{H4lem}
$H^4(M^7_{\bfn,\bfm},\bbz)$ contains $G_{reg}/G_{reg}^\mu$
\end{lemma}

We also have

\begin{lemma}\label{H3lem}
$H^3(M^7_{\bfn,\bfm},\bbz)=0.$
\end{lemma}

\begin{proof}
The only non-zero term of total degree 3 in \eqref{E2Leray} is
$$E^{2,1}_2=H^2(S_\bfn,R^1\pi_*\bbz)$$
where we have a differential 
$$d_2:E^{2,1}_2=H^2(S_\bfn,R^1\pi_*\bbz)\ra{2.0} E^{4,0}_2=H^4(S_\bfn,\bbz)=\bbz^3.$$
This differential gives torsion in $H^4(M^7_{\bfn,\bfm},\bbz)$ implying that $H^4(M^7_{\bfn,\bfm},\bbq)$ vanishes which in turn implies that $E^{2,1}_r=0$ for $r>2$ which proves the lemma.                                                            
\end{proof}

\begin{theorem}\label{7manthm}
Let $M^7_{\bfn,\bfm}$ denote the total space of the principal $S^1$ orbibundle over a KS orbifold $(S_\bfn,\grD_\bfm)$ defined by a primitive orbifold K\"ahler class $[\gro_{\bfn,\bfm}]=c_1x_1+c_2x_2+c_3x_3$ such that $\mu[\gro_{\bfn,\bfm}]$ is a primitive class in $H^2(S_\bfn,\bbz)$.
Then 
$$H^{r}(M^7_{\bfn,\bfm},\bbz)=\begin{cases} \bbz &~~\text{if $r=0,7$} \\
                                     \bbz^2&  ~~\text{if $r=2,5$} \\
                                     G &  ~~\text{if $r=4$} \\
                                     0 & ~~\text{otherwise.}
                                  \end{cases}$$
where $G$ is an Abelian group that contains $G_{reg}/G^\mu_{reg}$. In particular, if $\gcd(|G_{reg}|,\mu)=1$ then $G$ contains $G_{reg}$.
\end{theorem}

\begin{proof}
Since $[\gro_{\bfn,\bfm}]$ is a primitive orbifold class and $(S_\bfn,\grD_\bfm)$ has $\pi_1^{orb}(S_\bfn,\grD_\bfm)=\BOne$, then the orbifold $M^7_{\bfn,\bfm}$ has $\pi_1^{orb}(M^7_{\bfn,\bfm})=\BOne$ by the long exact homotopy sequence of the classifying spaces. The remainder of the theorem then follows from lemmas \ref{Qiso}, \ref{H4lem}, \ref{H3lem}.
\end{proof}

\begin{remark}
It follows from Theorem \ref{regs1bun} that $G_{reg}$ is never the identity; although,  we can have $G_{reg}=G_{reg}^\mu$ so that $G_{reg}/G^\mu=\BOne$ in this case. However, this does not imply that $H^4(M^7_{\bfn,\bfm},\bbz)$ is the identity. 
\end{remark}

\begin{remark}\label{dec-indec}
Let $M$ be a Boothby-Wang  Sasaki manifold over $\bbc\bbp^1\times\bbc\bbp^1$.
The set of KS orbifolds $\calk\cals_0$ splits into a disjoint union two types; those that have a polarization such that the corresponding Boothby-Wang constructed Sasaki orbifold is represented by a quasi-regular ray in the $\bfw$-cone of a join of the form $M \star_{\bfl} S^3_{\bfw}$, and those that do not have such a polarization. The former are characterized by the condition $n_1n_2>0$ (see Theorem \ref{whenS3join} below) and the corresponding Sasaki manifolds are called {\it cone decomposable} (see Definitions 3.1 and 4.1 in \cite{BHLT16}). It remains an open question whether a Boothby-Wang constructed Sasaki orbifold over the latter type, with $n_1n_2<0$, is necessarily {\it cone indecomposable} or simply not represented by a quasi-regular ray in the $\bfw$-cone of a join of the form $M \star_{\bfl} S^3_{\bfw}$.
\end{remark}

\subsection{When $M^7$ is an $S^3_\bfw$-join}\label{joinsect}
Generally, we are interested in how $G$ depends on $\bfm$ and $\bfn$.
Unfortunately, Theorem \ref{7manthm} does not give us much useful information about this dependence. However, we can determine this dependence in a particular cone decomposable case, namely when the space $M^7_{\bfn,\bfm}$ can be represented as a join of the form $M_{\bfl,\bfw}=M\star_\bfl S^3_\bfw$ where $M$ is a principal $S^1$ bundle over $\bbc\bbp^1\times \bbc\bbp^1$. Let $M^5$ denote the total space of the Boothby-Wang bundle over $\bbc\bbp^1\times \bbc\bbp^1$ determined by the K\"ahler class $k_1y_1+k_2y_2$ with $k_1,k_2\in \bbz^+$. When $k_1$ and $k_2$ are relatively prime, $M^5$ is diffeomorphic to $S^2\times S^3$. More generally $S^2\times S^3$ is an $k$-fold cover of $M^5$, namely we have $M^5\approx (S^2\times S^3)/\bbz_k$ where $k=\gcd(k_1,k_2)$. Moreover, if $\gcd(l_\infty,w_0)=\gcd(l_\infty,w_\infty)=1$, $M_{\bfl,\bfw}$ will be smooth, and we have an $S^1$ bundle
$$S^1\ra{2.6} M^5\times S^3_\bfw\ra{2.6} M^5\star_\bfl S^3_\bfw$$
which gives the homotopy exact sequence
\begin{equation}\label{homexjoin2}
0\ra{2.0}\pi_2(M^5)\ra{2.0}\pi_2(M^5\star_\bfl S^3_\bfw)\ra{2.0}\bbz\ra{2.0}\bbz_k \ra{2.0}\pi_1(M^5\star_\bfl S^3_\bfw)\ra{2.0}1.
\end{equation}
So when $k=1$ so $M^5=S^2\times S^3$, the join $M^7_{\bfl,\bfw}$ will also be simply connected with $\pi_2(M^7_{\bfl,\bfw})=\bbz^2$. More generally, the topological analysis proceeds as in Section 4 of \cite{BoTo14a}, see also \cite{BoTo18c}. We consider the commutative diagram of fibrations
\begin{equation}\label{orbifibrationexactseq}
\begin{matrix}
(S^2\times S^3)/\bbz_k\times S^3_\bfw &\ra{2.6} &M_{\bfl,\bfw}&\ra{2.6}
&\mathsf{B}S^1 \\
\decdnar{=}&&\decdnar{}&&\decdnar{\psi}\\
(S^2\times S^3)/\bbz_k \times S^3_\bfw&\ra{2.6} & \bbc\bbp^1\times\bbc\bbp^1\times\mathsf{B}\bbc\bbp^1[\bfw]&\ra{2.6}
&\mathsf{B}S^1\times \mathsf{B}S^1\, 
\end{matrix} \qquad \qquad
\end{equation}
where $\mathsf{B}G$ is the classifying space of a group $G$ or Haefliger's classifying space \cite{Hae84} of an orbifold if $G$ is an orbifold. Note that $H^2((S^2\times S^3)/\bbz_k,\bbz)=\bbz\oplus \bbz_k$. The diagram gives a map of spectral sequences, and we note that the lower fibration is a product of well understood fibrations. The $E_2$ term of the top fibration is 
\begin{equation}\label{E_2}
E_2^{p,q}=H^p(\mathsf{B}S^1,H^q(S^{2}\times S^3\times S^3_\bfw,\bbz))\approx \bbz[s]\otimes\grL[\gra,\grb,\grg],
\end{equation}
where $\gra$ is a 2-class, $\grb,\grg$ are 3-classes, and $s_1,s_2$ are the positive generators of $H^*(\mathsf{B}S^1,\bbz)$. By the Leray-Serre Theorem this converges to $H^{p+q}(M_{\bfl,\bfw},\bbz)$. The non-vanishing differentials of the bottom product fibration are, first $d_4(\grb)=s_1^2$ and second $d_4(\grg)=w_0w_\infty s_2^2$ and those induced by naturality. As described in \cite{BoTo14a} we also have $\psi^*s_1=l_\infty s$ and $\psi^*s_2=-l_0 s$. So by naturality the differentials of the top fibration are  $d_4(\grb)=l_\infty^2 s^2$ and $d_4(\grg)=w_0w_\infty l_0^2 s^2$. Now $M_{\bfl,\bfw}$ is smooth if and only if $\gcd(l_\infty,w_0w_\infty)=1$. This gives $H^4(M_{\bfl,\bfw},\bbz)=\bbz_{l_\infty^2}\times \bbz_{w_0w_\infty l_0^2}=\bbz_{w_0w_\infty l_0^2l_\infty^2}$.

\section{Admissible Projective Bundles}\label{adsect}
{\it Admissible projective bundles} were described in general \cite{ACGT08}. Here we will restrict ourselves to a specific type of these, namely projective bundles of the form 
$$S_\bfn=\bbp(\BOne \oplus \calo(n_1,n_2))\ra{1.8} \bbc\bbp^1\times\bbc\bbp^1$$ 
that satisfy the following conditions:
\begin{itemize}
\item For $i=1,2$ let $(\pm \gro_i,\pm g_i)$ be K\"ahler metrics with constant scalar curvature $\pm 2s_i=\pm 4/n_i$ on $\bbc\bbp^1$. [It is assumed that $n_i\neq 0$.] The $\pm$ means that either $+g_i$ or $-g_i$ is positive definite.
\item $\calo(n_1,n_2)\ra{1.8} \bbc\bbp^1\times\bbc\bbp^1$ is the holomorphic line bundle over $\bbc\bbp^1\times\bbc\bbp^1$ which satisfy
$$c_1(\calo(n_1,n_2))=\frac{1}{2\pi}[\gro_1 +\gro_2]=n_1 x_1+n_2 x_2$$
\end{itemize}

On admissible projective bundles such as the ones defined above we can now construct the {\it admissible metrics},  \cite{ACGT08}. Here we recover the main points of this construction.
 
Consider the standard circle action
on $S_\bfn=\bbp(\BOne \oplus \calo(n_1,n_2))\ra{1.8} \bbc\bbp^1\times\bbc\bbp^1$. It extends to a holomorphic
$\bbc^*$ action. The open and dense set $S_\bfn^0$ of stable points with respect to the
latter action has the structure of a principal circle bundle over the stable quotient.
The hermitian norm on the fibers induces via a Legendre transform a function
$\gz:S_\bfn^0\rightarrow (-1,1)$ whose extension to $S_\bfn$ consists of the critical manifolds
$e_0:= \gz^{-1}(1)=\bbp(\BOne \oplus 0)$ and $e_\infty:= \gz^{-1}(-1)=\bbp(0 \oplus \calo(n_1,n_2))$.
Letting $\theta$ be a connection one form for the Hermitian metric on $S_\bfn^0$, with curvature
$d\theta = \omega_1 +\omega_2$, an admissible K\"ahler metric and form are
given (up to scale) by the respective formulas
\begin{equation}\label{g}
g=\frac{1+r_1\gz}{r_1}g_1 + \frac{1+r_2\gz}{r_2}g_2+\frac {d\gz^2}
{\Theta (\gz)}+\Theta (\gz)\theta^2,\quad
\omega = \frac{1+r_1\gz}{r_1}\omega_{1}+ \frac{1+r_2\gz}{r_2}\omega_{2} + d\gz \wedge
\theta,
\end{equation}
valid on $S_\bfn^0$. Here $\Theta$ is a smooth function with domain containing
$(-1,1)$ and $r_i$, $i =1,2$ are real numbers of the same sign as
$g_{i}$ and satisfying $0 < |r_i| < 1$. The complex structure yielding this
K\"ahler structure is given by the pullback of the base complex structure
along with the requirement $Jd\gz = \Theta \theta$. The function $\gz$ is hamiltonian
with $K= J \,grad\, \gz$ a Killing vector field. In fact, $\gz$ is the moment 
map on $S_\bfn$ for the circle action, decomposing $M$ into 
the free orbits $S_\bfn^0 = \gz^{-1}((-1,1))$ and the special orbits 
$\gz^{-1}(\pm 1)$. Finally, $\theta$ satisfies
$\theta(K)=1$.

Now $g$ is a (positive definite) K\"ahler metric which extends smoothly to all of $S_\bfn$ if and only if
$\Theta$ satisfies the following positivity and boundary
conditions
\begin{align}
\label{positivity}
(i)\ \Theta(\gz) > 0, \quad -1 < \gz <1,\quad
(ii)\ \Theta(\pm 1) = 0,\quad
(iii)\ \Theta'(\pm 1) = \mp 2.
\end{align}

The K\"ahler class $\Omega_{\bfr} = [\omega]$ of an admissible metric as in \eqref{g} is also called
{\it admissible} and is uniquely determined by the parameters
$r_1, r_2$, once the data associated with $M$ (i.e.
$s_i=2/n_i$, $g_i$ etc.) is fixed. Indeed, we have
$$ \Omega_{\bfr} = \frac{[\omega_1]}{r_1} + \frac{[\omega_2]}{r_2} + 2\pi\Xi,$$
where $\Xi$ is the Poincare dual of $e_0+e_\infty$.
For a more thorough description of $\Xi$, please consult Section 1.3 of \cite{ACGT08}.
Note that on $S_\bfn$ any K\"ahler class is admissible up to scale.
Using that $y_3-x_3 = n_1x_1+n_2 x_2$, we can also write
$$
\begin{array}{ccl}
\Omega_{\bfr}/2\pi & = & \frac{n_1}{r_1} x_1+\frac{n_2}{r_2} x_2 + x_3+y_3  \\
\\
& = & \frac{n_1(1+r_1)}{r_1} x_1+\frac{n_2(1+r_2)}{r_2} x_2 +2x_3\\
\\
&=& \frac{n_1(1-r_1)}{r_1} x_1+\frac{n_2(1-r_2)}{r_2} x_2 +2y_3.
\end{array}
$$

Define a function $F(\gz)$ by the formula $\Theta(\gz)=F(\gz)/p_c(\gz)$, where
$p_c(\gz) =(1 + r_1 \gz)(1 + r_2 \gz)$.
Since $p_c(\gz)$ is positive for $-1\leq\gz\leq1$, conditions
\eqref{positivity}
are equivalent to the following conditions on $F(\gz)$.

\begin{align}
\label{positivityF}
(i)\ F(\gz) > 0, \quad -1 < \gz <1,\quad
(ii)\ F(\pm 1) = 0,\quad
(iii)\ F'(\pm 1) = \mp 2p_c(\pm1).
\end{align}

\subsection{Orbifolds} 
Now we allow our admissible metrics to compactify as orbifold metrics on the log pair
$$(S_\bfn,\Delta_\bfm)=\left(\bbp(\BOne \oplus \calo(n_1,n_2),\Delta_\bfm \right)\ra{1.8} \bbc\bbp^1\times\bbc\bbp^1,$$
where 
$$\Delta_\bfm = (1-1/m_0) D_1 + (1-1/m_\infty) D_2 = (1-1/m_0) e_0 + (1-1/m_\infty) e_\infty,$$
and $m_0,m_\infty \in \bbz^+$.
Then \eqref{positivity} generalizes to
\begin{equation}
\label{positivityThetaorbifold}
\begin{array}{cl}
(i)& \Theta(\gz) > 0, \quad -1 < \gz <1,\quad \\
(ii)& \Theta(\pm 1) = 0,\quad \\
(iii)& \Theta'(-1) =  2/m_\infty,\quad\text{and}\quad \Theta'(1) =  -2/m_0.
\end{array}
\end{equation}
and, with $\Theta(\gz)=F(\gz)/p_c(\gz)$ as above, we get that this is equivalent to
\begin{equation}
\label{positivityForbifold}
\begin{array}{cl}
(i)& F(\gz) > 0, \quad -1 < \gz <1,\quad \\
(ii)& F(\pm 1) = 0,\quad \\
(iii)& F'(-1) =  2p_c(-1)/m_\infty,\quad\text{and}\quad F'(1) =  -2p_c(1)/m_0.
\end{array}
\end{equation}

Note that this does not change the expression for $\Omega_r$ above whereas from \eqref{c1orb2} we already know that the adjusted Chern class is

$$
\begin{array}{ccl}
c_1^{orb}(S_\bfn,\grD_\bfm) 
                                             &=& \Bigl(2-\frac{n_1}{m_\infty}\Bigr)x_1 +\Bigl(2-\frac{n_2}{m_\infty}\Bigr)x_2+\Bigl(\frac{1}{m_0}+\frac{1}{m_\infty}\Bigr)y_3 \\
                                             &=& \Bigl(2+\frac{n_1}{m_0}\Bigr)x_1 +\Bigl(2+\frac{n_2}{m_0}\Bigr)x_2+\Bigl(\frac{1}{m_0}+\frac{1}{m_\infty}\Bigr)x_3.
                                             \end{array}.$$
                                             
\subsection{Connection with $S^3_\bfw$-joins}
Consider $(S_\bfn,\Delta_\bfm)=\left(\bbp(\BOne \oplus \calo(n_1,n_2),\Delta_\bfm \right)\ra{1.8} \bbc\bbp^1\times\bbc\bbp^1,$ polarized with a primitive orbifold K\"ahler class $[\omega_{\bfn,\bfm}]$. As mentioned above, this will be the rescale of some admissible K\"ahler class $\Omega_\bfr$ determined by $\bfr = (r_1,r_2)$. If $n_1n_2>0$, we will say that $[\omega_{\bfn,\bfm}]$ is {\em diagonally admissible} if and only if $r_1=r_2$. Note that if $n_1n_2<0$, it is never possible to have $r_1=r_2$, since the sign of $r_i$ is equal to the sign of $n_i$. Note that diagonally admissible is equivalent to being admissible (up to scale) as defined (more narrowly) in Section 2.5.1 of \cite{BoTo19a} with $N = \bbc\bbp^1\times\bbc\bbp^1$
and $[\omega_N] = \frac{n_1}{n} x_1+\frac{n_2}{n}x_2$, where $n=(\text{sign of}\,n_i)\gcd(|n_1|,|n_2|)$.

Indeed, suppose $n_1n_2>0$ and $r_1=r_2=r$ in the above setting. Then the admissible metric simplifies to
\begin{equation}\label{gdeg}
g=\frac{1+r\gz}{r}g_{N_n} +\frac {d\gz^2}
{\Theta (\gz)}+\Theta (\gz)\theta^2,\quad
\omega = \frac{1+r\gz}{r}\omega_{N_n}+ d\gz \wedge
\theta,
\end{equation}
with admissible K\"ahler class
$$
\begin{array}{ccl}
\Omega_{\bfr}/2\pi & = & \frac{n}{r}\left(\frac{n_1}{n} x_1+\frac{n_2}{n} x_2 \right)+ x_3+y_3  \\
\\
& = & \frac{n(1+r)}{r} \left(\frac{n_1}{n} x_1+\frac{n_2}{n} x_2 \right) +2x_3\\
\\
&=&  \frac{n(1-r)}{r} \left(\frac{n_1}{n} x_1+\frac{n_2}{n} x_2 \right) +2y_3,
\end{array}
$$
where $g_{N_n} = n_1g_1+n_2g_2$.
As such, we can view this metric as an admissible metric on $S_\bfn=\bbp(\BOne \oplus L_n)\ra{1.8} N$, where $N=\bbc\bbp^1\times\bbc\bbp^1$
and 
\begin{itemize}
\item $(\pm \gro_{N_n},\pm g_{N_n})$ is a K\"ahler metric with constant scalar curvature $\pm 4s=\pm 4(\frac{1}{n_1} + \frac{1}{n_2})$. [The $\pm$ still means that either $+g_{N_n}$ or $-g_{N_n}$ is positive definite.]
\item $L_n\ra{1.8} N$ is the holomorphic line bundle over $N=\bbc\bbp^1\times\bbc\bbp^1$ which satisfy
$$c_1(L_n)=\frac{1}{2\pi}[\gro_{N_n}]=n \left(\frac{n_1}{n} x_1+\frac{n_2}{n} x_2 \right).$$ 
\end{itemize}
This special case, is exactly of the type considered in e.g. Section 5 of \cite{BoTo14a} or more generally (including cyclic orbifolds $N$) in Section 2.5 of \cite{BoTo19a}.
Note that the base metric $(\pm \gro_{N_n},\pm g_{N_n})$ is only K\"ahler-Einstein if also $n_1=n_2$.

Now consider the natural Boothby-Wang constructed Sasaki structure on the $S^1$-bundle $M \rightarrow N$ defined by the primitive
$[\omega_N]=\left(\frac{n_1}{n} x_1+\frac{n_2}{n} x_2 \right)$.
Following Proposition 4.22 of \cite{BHLT16} (with the amendment of Lemma 2.12 and Corollary 2.14 of \cite{BoTo19a}) we then realize the following connection with $S^3_\bfw$-joins.
\begin{proposition}\label{goingback}
For $n_1n_2>0$, $r_1=r_2=r$ rational, $n=(\text{sign of}\,n_i)\gcd(|n_1|,|n_2|)$, and $\gcd(m_1,m_2,|n|)=1$, there is a choice of co-prime $w_0$, $w_\infty \in \bbz^+$ and co-prime $l_0, l_\infty$ such that, when we form the $S^3_\bfw$-join $M_{l_0,l_\infty,\bfw} := M \star_{l_0,l_\infty} S^3_{w_0,w_\infty} = M \star_{\bfl} S^3_{\bfw}$, the quasi-regular quotient of $M_{l_0,l_\infty,\bfw}$ by the flow of the Reeb vector field
$\xi_\bfv$ determined by $(v_0,v_\infty)$ in the $\bfw$-cone (where $m=\gcd(m_0,m_\infty)$ and $\bfm= (m_0,m_\infty)=m\bfv=m(v_0,v_\infty)$), is the log pair
$(S_\bfn,\Delta_\bfm)$ with induced (transverse) K\"ahler class $m\gcd(\gs, w_0 v_\infty )  [\omega_{\bfn,\bfm}]$ where 
$\gs=\gcd(l_\infty, |w_0v_\infty-w_\infty v_0|)$ and
$[\omega_{\bfn,\bfm}]$ is a orbifold primitive class that is an appropriate rescale of the admissible class
$\Omega_\bfr$ on $(S_\bfn,\Delta_\bfm)$. The join is smooth if and only if $\gcd(w_0,l_\infty)=\gcd(w_\infty, l_\infty)=1$. In particular, the join is smooth if $m_0=m_\infty=1$.
\end{proposition}
\begin{proof}
We follow the proof of Proposition 4.22 in \cite{BHLT16} (including the paragraphs leading up to the proposition),
to identify the join by first picking $\bfw = (w_0,w_\infty)$ to be the unique positive, integer, and co-prime solution of
$$r=\frac{w_0 m_\infty - w_\infty m_0}{w_0 m_\infty +w_\infty m_0}$$
and then, as the next step, picking the pair $(l_0,l_\infty)$ as the unique, positive integers, and co-prime solution of
$$l_\infty n= l_0(w_0m_\infty - w_\infty m_0).$$

Proposition 4.22 in \cite{BHLT16} was slightly misleading in implying that the quasi-regular quotient of $\xi_\bfv$ would always produce a primitive orbifold K\"ahler class, but, with Lemma 2.12 and Corollary 2.14 of \cite{BoTo19a} in hand, we can say that the transverse K\"ahler class is $m\gcd(\gs\Upsilon_N, w_0 v_\infty l_0) [\omega_{\bfn,\bfm}]$ where $[\omega_{\bfn,\bfm}]$ is a orbifold primitive class that is an appropriate rescale of the admissible class
$\Omega_\bfr$ on $(S_\bfn,\Delta_\bfm)$ and $\Upsilon_N$ is the orbifold order of $N$. Since $\Upsilon_N=1$ and $\gcd(\gs,l_0)=1$ (recall $\gcd(l_0,l_\infty)=1$) we get the desired K\"ahler class.

The rest of the claims follow straight from the proof Proposition 4.22 in \cite{BHLT16}. In particular, note that $m=l_\infty/\gs$, as it should be according to Theorem 3.8
in \cite{BoTo14a}.
\end{proof}

\begin{remark}
Note that when the join is smooth, it is easy to see that $\gcd(\gs, w_0 v_\infty)=1$ and so the transverse K\"ahler class is $m [\omega_{\bfn,\bfm}]$. In particular, if
$m_0=m_\infty=1$, we do have a primitive (and admissible) transverse K\"ahler class.
\end{remark}

Combining Theorem 2.7 and Lemma 2.12 from \cite{BoTo19a} for the``only if " and Proposition \ref{goingback} for the ``if'',  we arrive at the following theorem.
\begin{theorem}\label{whenS3join}
Assume $\gcd(m_1,m_2,|n_1|,|n_2|)=1$ and 
$[\omega_{\bfn,\bfm}]$ is a orbifold primitive class on $(S_\bfn,\Delta_\bfm)$.
There exist a constant $k\in \bbz^+$ such that the polarized orbifold $(S_\bfn,\Delta_\bfm, k[\omega_{\bfn,\bfm}])$ is the quotient with respect to (the canonical Reeb vector field in) a quasi-regular ray in the $\bfw$-cone of a (possibly non-smooth) join of the form $M \star_{\bfl} S^3_{\bfw}$, where $M$ is a Boothby-Wang constructed Sasaki manifold over $\bbc\bbp^1\times\bbc\bbp^1$, if and only if $n_1n_2>0$ and $[\omega_{\bfn,\bfm}]$ is a diagonally admissible K\"ahler class.
\end{theorem}

Suppose now that $n_1n_2>0$, we are in the log Fano case as in Proposition \ref{orbFano}, and the polarization is chosen such that $c_1^{orb}(S_\bfn,\grD_\bfm) =\cali_{\bfn,\bfm} [\omega_{\bfn,\bfm}]$.
Since $c_1^{orb}(S_\bfn,\grD_\bfm) = \Bigl(2-\frac{n_1}{m_\infty}\Bigr)x_1 +\Bigl(2-\frac{n_2}{m_\infty}\Bigr)x_2+\Bigl(\frac{1}{m_0}+\frac{1}{m_\infty}\Bigr)y_3$
and a diagonally admissible K\"ahler class is a rescale of $\Omega_\bfr = \frac{(1-r)}{r} \left(n_1 x_1+n_2 x_2 \right) +2y_3$ for some $0<|r|<1$ with $rn_i>0$, we see that $[\omega_{\bfn,\bfm}]=c_1^{orb}(S_\bfn,\grD_\bfm) /\cali_{\bfn,\bfm}$ is diagonally admissible if and only if 
$$
2-\frac{n_1}{m_\infty}=\left(\frac{1}{m_0}+\frac{1}{m_\infty}\right) \frac{n_1(1-r)}{2r} \quad \text{and}\quad
2-\frac{n_2}{m_\infty}=\left(\frac{1}{m_0}+\frac{1}{m_\infty}\right) \frac{n_2(1-r)}{2r},
$$
i.e.,
$$
n_1\left(\frac{1}{m_\infty}+\left(\frac{1}{m_0}+\frac{1}{m_\infty}\right) \frac{(1-r)}{2r} \right)=
2=n_2\left(\frac{1}{m_\infty}+\left(\frac{1}{m_0}+\frac{1}{m_\infty}\right) \frac{(1-r)}{2r}\right).
$$
This clearly implies that we must have $n_1=n_2$. On the other hand, if $n_1=n_2$ (and still assuming log Fano) we can solve for an appropriate $r$. In conclusion,
\begin{corollary}
In the log Fano case, there exists a constant $k\in \bbz^+$ such that the polarized orbifold $(S_\bfn,\Delta_\bfm, kc_1^{orb}(S_\bfn,\grD_\bfm) )$ is the quotient with respect to a quasi-regular ray in the $\bfw$-cone of a (possibly non-smooth) join of the form $M \star_{\bfl} S^3_{\bfw}$, where $M$ is a Boothby-Wang constructed Sasaki manifold over $\bbc\bbp^1\times\bbc\bbp^1$, if and only if $n_1=n_2$.
\end{corollary} 

Returning to Proposition \ref{goingback} and Theorem \ref{whenS3join}, we note that in the case where $(S_\bfn,\Delta_\bfm, [\omega_{\bfn,\bfm}])$ is the quotient with respect to a quasi-regular ray in the $\bfw$-cone of a join of the form $M \star_{\bfl} S^3_{\bfw}$  (so assuming $k=1$ in Theorem \ref{whenS3join}, i.e., assuming the transverse K\"ahler class is primitive) we must have that $m=1$ and $\gcd(\gs, w_0 v_\infty )=1$. The first equality implies that $l_\infty=\gs=\gcd(l_\infty, |w_0v_\infty-w_\infty v_0|)$ and hence
$l_\infty$ is a factor of $|w_0v_\infty-w_\infty v_0|$. Since we also have that $\gcd(l_\infty, w_0 v_\infty )=1$, we must have $\gcd(l_\infty, w_\infty v_0 )=1$ and in particular  
$\gcd(w_0,l_\infty)=\gcd(w_\infty, l_\infty)=1$. This means that $M \star_{\bfl} S^3_{\bfw}$ is smooth. We therefore have the following companion to Theorem \ref{whenS3join}.

\begin{proposition}\label{whenS3join2}
Consider the polarized orbifold $(S_\bfn,\Delta_\bfm, [\omega_{\bfn,\bfm}])$, where $[\omega_{\bfn,\bfm}]$ is a primitive integer orbifold K\"ahler class and
$\gcd(m_0,m_\infty,|n_1|,|n_2|)=1$.

If $(S_\bfn,\Delta_\bfm, [\omega_{\bfn,\bfm}])$ is the quotient with respect to (the canonical Reeb vector field in) a quasi-regular ray in the $\bfw$-cone of a join of the form $M \star_{\bfl} S^3_{\bfw}$, then this join is smooth. Moreover, in this case $\gcd(m_0,m_\infty)=1$.
\end{proposition}

\subsection{Connection with Yamazaki's Fiber joins}\label{yam}
In \cite{Yam99} T. Yamazaki introduced the fiber join for $K$-contact structures and in \cite{BoTo18b} this was extended to Sasaki structures.
In the smooth case ($m_0=m_\infty=1$) of KS orbifolds, it follows from Section 5.3 of \cite{BoTo18b} that 
$S_\bfn$ polarized by a primitive K\"ahler class, which in turn is an appropriate rescale of $\Omega_\bfr$, is
the quotient of the regular ray in the $\gt^+_{sphr}$ cone of a Yamazaki fiber join over $\bbc\bbp^1\times\bbc\bbp^1$, if an only if
there exist $k_1^1, k_1^2, k_2^1, k_2^2 \in \bbz^+$ such that
\begin{eqnarray}\label{yamequation}
k_1^1-k_2^1& = & n_1\qquad k_1^2-k_2^2=n_2  \notag \\
\frac{k_1^1-k_2^1}{k_1^1+k_2^1}&=&r_1 \qquad \frac{k_1^2-k_2^2}{k_1^2+k_2^2} = r_2.
\end{eqnarray}
In that case, the corresponding Yamazaki fiber join is given as follows:\footnote{We refer to \cite{BoTo18b} for details and a more general description of the Yamazaki fiber joins.} 

Let $\pi_j: \bbc\bbp^1\times\bbc\bbp^1 \rightarrow \bbc\bbp^1$ denote the natural projection to the $j^{t}$ factor of the product $\bbc\bbp^1\times\bbc\bbp^1 \rightarrow \bbc\bbp^1$
and let $\calk$ denote the canonical bundle on $\bbc\bbp^1$.
Let $L_i$ be a holomorphic line bundle over $\bbc\bbp^1\times\bbc\bbp^1 \rightarrow \bbc\bbp^1$ given by
$L_i=k^1_i \pi_1^*\calk^{\frac{-1}{2}}+k^2_i\pi_2^*\calk^{\frac{-1}{2}}$, for $k^j_i \in \bbz^+$. The choice of $L_1$ and $L_2$ can be given by the matrix $K=\begin{pmatrix} k^1_1&k^2_1\\
k^1_2&k^2_2 \end{pmatrix}$. 
Note that $c_1(L_i)$ are both in the K\"ahler cone of $\bbc\bbp^1\times\bbc\bbp^1$ and hence $L_i$ are positive line bundles over $\bbc\bbp^1\times\bbc\bbp^1$.
Each $c_1(L_i)$ also defines a principal $S^1$-bundle over $\bbc\bbp^1\times\bbc\bbp^1$, $M_i \rightarrow \bbc\bbp^1\times\bbc\bbp^1$, and we identify $L_i$ with $M_i \times_{S^1} \bbc$. Then $M_i \stackrel{\pi}{\rightarrow} \bbc\bbp^1\times\bbc\bbp^1$ has a natural Sasaki structure defined by the Boothby-Wang construction. 
Consider now $L_1^*\oplus L_2^*$ and equip each $L_i^*$ with a Hermitian metric giving us a norm $d_i: L_i^* \to \bbr^{\geq 0}$.
Then, the fiber join, $M_K=M_1\star_f M_2$ is defined as the $S^3$-bundle over $S$ whose fibers are given by $d_1^2+d_2^2=1$.
Now $M_K$ has a natural $CR$-structure $(\cald, J)$ with a family, $\gt^+_{sphr}$, of compatible Sasaki structures $\cals_\bfa = (\xi_\bfa, \eta_\bfa, \Phi_\bfa, g_\bfa)$, 
where $\bfa = (a_1,a_2) \in (\bbr^+)^2$ and $(a_1,a_2)=(1,1)$ corresponds to the regular Sasaki structure in $\gt^+_{sph}(\cald, J)$. Note that $\gt^+_{sph}$ is a proper subcone of the unreduced Sasaki cone $\gt^+$ of $(M_K,\cald, J)$.

\begin{remark}
In \cite{BoTo18b} we did not consider the more general question  of determining the quotients of quasi-regular Sasaki structures in $\gt^+_{sph}$ and the transverse K\"ahler class. We conjecture that those would indeed be certain polarized KS orbifolds and will explore this in future studies.
\end{remark}

\subsection{Ricci Solitons and K\"ahler-Einstein}\label{KEsection}
It is well-known that there exists a unique K\"ahler-Ricci soliton on any toric compact Fano complex orbifold
\cite{Don08,TiZh00,ShZh12, WaZh04,Zhu00}. The existence proofs by Wang-Zhu \cite{WaZh04}, Shi-Zhu \cite{ShZh12} and Donaldson \cite{Don08} all use a continuity method and thus do not provide an explicit expression of the K\"ahler-Ricci soliton. It is therefore interesting to explore cases where explicit descriptions of K\"ahler-Ricci solitons are possible. Explicit examples can be found in e.g.\cite{DaWa11,LeTo13}.

Here we want to explore the Ricci Soliton and more specifically the K\"ahler-Einstein equations 
under the constraint of the orbifold endpoint conditions \eqref{positivityForbifold}. This will yield explicit examples 
on Fano $(S_\bfn,\Delta_\bfm)$ and is
essentially a mild orbifold extension of the work by Koiso and Sakane \cite{KoSa86,Koi90}. We will follow in their footsteps using the notation of Section 3 of \cite{ACGT08b}.

The admissible metrics are Ricci solitons with $V= (\frac{c}{2})\,grad_g \gz$ if and only if
\begin{equation}\label{Ricci-Soliton}
\rho - \lambda \omega = \call_V\omega,
\end{equation}
where the Ricci form is
$$\rho = s_1 \omega_1 + s_2 \omega_2 - \frac{1}{2}dd^c \log F =  \left(s_1-\frac{1}{2} \frac{F'(\gz)}{p_c(\gz)}\right) \omega_1 + \left(s_2-\frac{1}{2} \frac{F'(\gz)}{p_c(\gz)}\right) \omega_2 - \frac{1}{2} \left( \frac{F'}{p_c}\right)' (\gz)\,d\gz\wedge \theta $$
and $\lambda, c\in \bbr$. Obviously, $c=0$ corresponds to K\"ahler-Einstein metrics.
Since $\call_V\omega =  (\frac{c}{2})\,dd^c\gz$ and $s_i=2/n_i$, equation \eqref{Ricci-Soliton} becomes a pair of ODEs
\begin{equation}\label{rs}
\begin{array}{ccl}
\frac{F'(\gz)}{p_c(\gz)} + c\frac{F(\gz)}{p_c(\gz)} & = & 4/n_1-2\lambda(\gz + \frac{1}{r_1})\\
\\
\frac{F'(\gz)}{p_c(\gz)} + c\frac{F(\gz)}{p_c(\gz)} & = & 4/n_2-2\lambda(\gz + \frac{1}{r_2}),\\
\end{array}
\end{equation}
Note that \eqref{rs} implies that $n_1=n_2$ if and only if $r_1=r_2$. 

Note that ({\it ii}) and ({\it iii}) of \eqref{positivityForbifold} together with \eqref{rs} implies the necessary conditions
\begin{equation}\label{lambda}
2\lambda = \frac{1}{m_0} + \frac{1}{m_\infty}
\end{equation}
and
\begin{equation}\label{rscondition}
r_1 = \frac{\frac{1}{m_0}+\frac{1}{m_\infty}}{\frac{4}{n_1}+\frac{1}{m_0}-\frac{1}{m_\infty}},\quad \quad r_2 = \frac{\frac{1}{m_0}+\frac{1}{m_\infty}}{\frac{4}{n_2}+\frac{1}{m_0}-\frac{1}{m_\infty}}.\end{equation}
On the other hand, assuming \eqref{lambda} and \eqref{rscondition}, the ODEs of \eqref{rs} are equivalent to the single ODE
\begin{equation}\label{rs2}
\begin{array}{ccl}
\frac{F'(\gz)}{p_c(\gz)} + c\frac{F(\gz)}{p_c(\gz)} & = & (\frac{1}{m_\infty} - \frac{1}{m_0}) - (\frac{1}{m_0}+\frac{1}{m_\infty})\gz
\end{array}
\end{equation}
Further, \eqref{rs2} and ({\it ii}) of  \eqref{positivityForbifold} together imply ({\it iii}) of \eqref{positivityForbifold}. 
In summary, we will get an admissible Ricci soliton solution as in \eqref{Ricci-Soliton} exactly when \eqref{lambda} and
 \eqref{rscondition} are satisfied and \eqref{rs2} has a solution $F(\gz)$ that satisfies ({\it i}) and ({\it ii}) of \eqref{positivityForbifold}.

Notice that since we require $0<|r_i|<1$ and $r_i n_i>0$, \eqref{rscondition} has an appropriate solution $(r_1,r_2)$ iff
$$\frac{n_1}{m_\infty}<2,\quad \frac{n_2}{m_\infty}<2,\quad -\frac{n_1}{m_0}<2,\quad -\frac{n_2}{m_0}<2.$$
This is exactly the Fano condition in Proposition \ref{orbFano}. In turn, under this condition, \eqref{lambda} and \eqref{rscondition} correspond exactly to
$c_1^{orb}(S_\bfn,\grD_\bfm) = \lambda \Omega_{\bfr}/2\pi$.

Similarly to the smooth case, we shall see that in the Fano case, there is indeed always an admissible Ricci soliton in the appropriate K\"ahler class:

Using the integrating factor method, observe that
\begin{equation} \label{Frssolution}
\begin{array}{ccl}
F(\gz) & = & e^{-c\,\gz}\int_{-1}^{\gz}e^{c\,t}\bigl((\frac{1}{m_\infty} - \frac{1}{m_0}) - (\frac{1}{m_0}+\frac{1}{m_\infty})t\bigr)p_c(t) dt\\
\\
&=& e^{-c\,\gz}\int_{-1}^{\gz}e^{c\,t}(t-t_{0})g(t) dt
\end{array}
\end{equation}
solves \eqref{rs2} and ({\it ii}) of  \eqref{positivityForbifold} iff $G(c)=0$, where
\begin{equation}\label{Gfunction}
G(k)= e^{k\,t_0}\int_{-1}^{1} e^{k\,(t-t_0)}(t-t_{0})g(t) dt,
\end{equation}
$t_{0} = \frac{m_0-m_\infty}{m_0+m_\infty}$, and $g(t) = -(\frac{1}{m_0} + \frac{1}{m_\infty})p_c(t)$. 

Note that $t_0  \in (-1,1)$ and  $g(t)<0$ for $t\in [-1,1]$. Thus $e^{-kt_{0}} G(k)$ is a strictly decreasing function of $k$ tending to
$\mp\infty$ as $k\to\pm\infty$, and hence has a unique zero $c$ (consistent with
the uniqueness of Ricci solitons). 

To check ({\it i}) of \eqref{positivityForbifold} we consider another auxiliary  function 
$$h(\gz) = e^{c\,\gz}F(\gz) = \int_{-1}^{\gz}e^{c\,t}(t-t_{0})g(t) dt.$$
Note that the sign of $h(\gz)$ equals the sign of $F(\gz)$.
In particular, $h(\pm1)=0$, and (due to ({\it iii}) of \eqref{positivityForbifold}) $h$ is positive to the immediate right of $\gz=-1$ and immediate left of $\gz=+1$.
Now, since $h'$ clearly has exactly one zero (namely $t_{0}$) in $(-1,1)$, it
is positive on $(-1,1)$. Therefore ({\it i}) of  \eqref{positivityForbifold} is also satisfied. 

Combining the arguments above with Proposition \ref{orbFano} we conclude the following.
\begin{proposition}
For $(S_\bfn,\Delta_\bfm)$  the following conditions are equivalent:
\begin{itemize}
\item The inequalities $\frac{n_1}{m_\infty}<2,\, \frac{n_2}{m_\infty}<2,\, -\frac{n_1}{m_0}<2,\, -\frac{n_2}{m_0}<2$ are satisfied;
\item $(S_\bfn,\Delta_\bfm)$ is log Fano;
\item There exist a K\"ahler-Ricci soliton on $(S_\bfn,\Delta_\bfm)$.
\end{itemize}
In this case, the K\"ahler-Ricci soliton $(g,\omega)$ is admissible and satisfies
$\rho=\call_V \omega,$
with $\lambda = \frac{1}{2m_0} + \frac{1}{2m_\infty}$ and
$V= (\frac{c}{2})\,grad_g \gz$ for a suitable real constant $c$. 
The K\"ahler-Ricci soliton is K\"ahler-Einstein iff this $c$ is equal to zero.
\end{proposition}
Of course this can also be seen as an orbifold extension of Theorem 3.1 in \cite{ACGT08} (which in turn is essentially due to Koiso \cite{Koi90})
in the case of $S_\bfn= \bbp(\BOne \oplus \calo(n_1,n_2))\ra{1.8} \bbc\bbp^1\times\bbc\bbp^1$.

\begin{remark}
Note that, assuming we are in the Fano case, $c$ in the above proposition is a non-zero multiple of the Futaki invariant, $\calf(V)$, of $c_1^{orb}(S_\bfn,\grD_\bfm)$ applied to $V$.
This underscores the well-known fact that the existence of a K\"ahler-Ricci Soliton with respect to a non-trivial holomorphic vector field is an obstruction to the existence of a K\"ahler-Einstein metric.
\end{remark}

\subsubsection{K\"ahler-Einstein}

Now we turn our attention to K\"ahler-Einstein examples on $(S_\bfn,\Delta_\bfm)$.
From the work above we see that under the assumption that the log Fano condition and \eqref{rscondition} hold, $\Omega_\bfr$ contains an admissible K\"ahler-Einstein metric
with $\rho=(\frac{1}{2m_0} + \frac{1}{2m_\infty})\omega$
if and only if $G$ defined in \eqref{Gfunction} satisfies that $G(0)=0$, i.e., 
\begin{equation}\label{ke2}
\int_{-1}^{1}\bigl((\frac{1}{m_\infty} - \frac{1}{m_0}) - (\frac{1}{m_0}+\frac{1}{m_\infty})t\bigr)(1+r_1 t)(1+r_2 t) dt =0.
\end{equation}

Carrying out the integration in \eqref{ke2} and substituting \eqref{rscondition}, we have the following
\begin{proposition}\label{KEDioPhan}
The Bott orbifold $(S_\bfn,\Delta_\bfm)$ admits a K\"ahler-Einstein metric (which happens to be admissible) iff the following two conditions are satisfied
\begin{enumerate}
\item $(S_\bfn,\Delta_\bfm)$ is log Fano (i.e., $\frac{n_1}{m_\infty}<2,\, \frac{n_2}{m_\infty}<2,\, -\frac{n_1}{m_0}<2,\, -\frac{n_2}{m_0}<2$)
\item $$24(m_0^3 m_\infty^2 - m_0^2 m_\infty^3) -8(n_1+n_2)(m_0^3 m_\infty  -  m_0^2 m_\infty^2  + 
   m_0 m_\infty^3 ) + 
  3 n_1 n_2(m_0^3  -  m_0^2 m_\infty  +  m_0 m_\infty^2  -  m_\infty^3 )=0$$
\end{enumerate}
\end{proposition}
\begin{example}\label{nonex}

If we assume $n_1=1$ and $n_2=2$, then the equation in Proposition \ref{KEDioPhan} rewrites to 
$$ 6 (2 m_0 m_\infty -m_0 + m_\infty ) (m_0^2(-1 + 2 m_\infty) - m_\infty^2(1 + 2 m_0))=0.$$
Clearly $2 m_0 m_\infty -m_0 + m_\infty =0$ has no positive integer solutions $(m_0,m_\infty)$.
Likewise, assume by contradiction that  $(m_0,m_\infty)$ are positive integer solutions of
$$m_0^2(-1 + 2 m_\infty) - m_\infty^2(1 + 2 m_0)=0,$$
i.e.,
$$m_0^2(-1 + 2 m_\infty) = m_\infty^2(1 + 2 m_0).$$
Since $gcd(2m_i\pm 1, m_i)=1$ for $i=0,\infty$ we see that this would imply that $m_0=m_\infty$ (since they would have to have the same prime factorization).
But $m_0=m_\infty$ will clearly never solve the equation.
\end{example}

\begin{remark}
The diophantine nature of the equation in Proposition \ref{KEDioPhan} makes it hard to spot solutions other than the classic smooth Koiso-Sakane example (see Example \ref{ksvariations} below). Further, Example \ref{nonex} and Proposition \ref{KEDioPhan} tells us that there exist at least one pair $\bfn = (n_1, n_2)$ 
such that for all pairs $\bfm=(m_0,m_\infty)\in \bbz^+\times \bbz^+$, the Bott orbifold $(S_\bfn,\Delta_\bfm)$ admits no K\"ahler-Einstein metric.
For these reasons we will instead view a choice of $(r_1,r_2) \in \bbq^2$ such that $0<|r_i|<1$ as a pair of parameters that determine a unique KE example on the appropriate corresponding KS orbifold. 
\end{remark}

From this point of view, note that \eqref{ke2} is equivalent to the equation
\begin{equation}\label{ke3}
\frac{m_\infty}{m_0} = \frac{3 +r_1 r_2 - r_1-r_2}{3+r_1 r_2+ r_1+r_2}.
\end{equation}
and, assuming \eqref{ke3}, equation \eqref{rscondition} is equivalent to
\begin{equation}\label{n1&n2}
\frac{n_1}{m_0} = \frac{2r_1(3 +r_1 r_2 - r_1-r_2)}{3+2 r_1 r_2+ r_1^2},\quad \frac{n_2}{m_0} = \frac{2r_2(3 +r_1 r_2 - r_1-r_2)}{3+2 r_1 r_2+ r_2^2}.
\end{equation}
We notice that if we pick a rational pair $(r_1,r_2)$ such that $0<|r_i|<1$, then there is a unique quadruple of appropriate integers $(n_1,n_2,m_0,m_\infty)$, solving
\eqref{ke3} and \eqref{n1&n2}, such that $n_i$ has the same sign as $r_i$, $m_i>0$, and $\gcd(|n_1|, |n_2|,m_0,m_\infty)=1$. 
This yields a KE example on the corresponding KS orbifold $(S_\bfn,\grD_\bfm)\ra{1.8} \bbc\bbp^1\times\bbc\bbp^1$ with fibers
$\bbc\bbp^1[v_0,v_\infty)]/\bbz_{m}$. Here we have used our earlier notation; $m=\gcd(m_0,m_\infty)$ and $\bfm= (m_0,m_\infty)=m\bfv=m(v_0,v_\infty)$.

\begin{example}\label{ksvariations}
Assume $r_1=r$ and $r_2=-r$ with $r \in (0,1)\cap \bbq$. Then \eqref{ke3} and \eqref{n1&n2} yields
$$\frac{m_\infty}{m_0} =1,\quad \frac{n_1}{m_0} = 2r, \quad \frac{n_2}{m_0} = -2r.$$
When $r=1/2$ this yields the Koiso-Sakane smooth K\"ahler-Einstein metric on $\bbp(\BOne \oplus \calo(1,-1)) \rightarrow
\bbc\bbp^1 \times \bbc\bbp^1$ (\cite{KoSa86}). In general, we write $r=p/q$ in reduced form with co-prime integers $0<p<q$.
Then
$$(n_1,n_2,m_0,m_\infty) = \left\{\begin{array}{cl} (2p,-2p,q,q) & \text{if}\,\, q\,\, \text{is odd}\\
\\
 (p,-p, q/2,q/2) & \text{if}\,\, q\,\, \text{is even}
 \end{array}\right. $$
This gives us K\"ahler-Einstein metrics on $\bbp(\BOne\oplus \calo(2p,-2p))$ with fibers
$\bbc\bbp^1/\bbz_q$ and on $\bbp(\BOne\oplus \calo(p,-p))$ with fibers
$\bbc\bbp^1/\bbz_{q/2}$, respectively. Using \eqref{index} we see that the index is  equal to $2$ if $q$ is odd and equal to $1$ if $q$ is even.

\end{example}
 
Below we will explore more examples beyond the Koiso-Sakane examples.
We set $r_i=p_i/q_i$ with $q_i \in \bbz^{\geq 2}$ and $p_i\in \bbz$ such that $0<|p_i|<q_i$ and $\gcd(|p_i|,q_i)=1$. Further, without loss we can assume that 
$r_1$, hence $p_1$, is positive. We then observe that \eqref{ke3} and \eqref{n1&n2} are equivalent to

\begin{equation}\label{solnfam}
\begin{array}{ccl}
\frac{m_\infty}{m_0} & = & \frac{3q_1 q_2 +p_1 p_2 - p_1q_2-p_2q_1}{3q_1 q_2+p_1 p_2+ p_1q_2+p_2q_1}\\
\\
\frac{n_1}{m_0} &= & \frac{2p_1(3q_1q_2 +p_1 p_2 - p_1q_2-p_2q_1)}{3q_1^2q_2+2 p_1 p_2q_1+ p_1^2q_2}\\
\\
\frac{n_2}{m_0} & = & \frac{2p_2(3q_1q_2 +p_1 p_2 - p_1q_2-p_2q_1)}{3q_1q_2^2+2 p_1 p_2q_2+ p_2^2q_1}.
\end{array}
\end{equation}

Note that, due to $0<|p_i|<q_i$, the denominator of each the right hand fractions in \eqref{solnfam} are all positive. Further, the numerator of the first two right hand fractions are positive while the numerator of the last right hand fraction has the same sign as $p_2$. Thus if we set
 $$k =\gcd \left(\begin{array}{r}
(3q_1 q_2+p_1 p_2+ p_1q_2+p_2q_1) \left(3q_1^2q_2+2 p_1 p_2q_1+ p_1^2q_2\right) \left(3q_1q_2^2+2 p_1 p_2q_2+ p_2^2q_1\right),\\
(3q_1 q_2 +p_1 p_2 - p_1q_2-p_2q_1) \left(3q_1^2q_2+2 p_1 p_2q_1+ p_1^2q_2\right) \left(3q_1q_2^2+2 p_1 p_2q_2+ p_2^2q_1\right),\\
2p_1(3q_1q_2 +p_1 p_2 - p_1q_2-p_2q_1)(3q_1 q_2+p_1 p_2+ p_1q_2+p_2q_1) \left(3q_1q_2^2+2 p_1 p_2q_2+ p_2^2q_1\right),\\
2|p_2|(3q_1q_2 +p_1 p_2 - p_1q_2-p_2q_1) (3q_1 q_2+p_1 p_2+ p_1q_2+p_2q_1) \left(3q_1^2q_2+2 p_1 p_2q_1+ p_1^2q_2\right)
\end{array}\right)$$
we have an appropriate solution $(n_1,n_2,m_0,m_\infty)$ to \eqref{solnfam} given by
\begin{equation}\label{solnfam2}
\begin{array}{ccl}
n_1&=&\frac{2p_1(3q_1q_2 +p_1 p_2 - p_1q_2-p_2q_1)(3q_1 q_2+p_1 p_2+ p_1q_2+p_2q_1) \left(3q_1q_2^2+2 p_1 p_2q_2+ p_2^2q_1\right)}{k}\\
n_2&=&\frac{2p_2(3q_1q_2 +p_1 p_2 - p_1q_2-p_2q_1) (3q_1 q_2+p_1 p_2+ p_1q_2+p_2q_1) \left(3q_1^2q_2+2 p_1 p_2q_1+ p_1^2q_2\right)}{k}\\
m_0&=&\frac{(3q_1 q_2+p_1 p_2+ p_1q_2+p_2q_1) \left(3q_1^2q_2+2 p_1 p_2q_1+ p_1^2q_2\right) \left(3q_1q_2^2+2 p_1 p_2q_2+ p_2^2q_1\right)}{k}\\
m_\infty&=&\frac{(3q_1 q_2 +p_1 p_2 - p_1q_2-p_2q_1) \left(3q_1^2q_2+2 p_1 p_2q_1+ p_1^2q_2\right) \left(3q_1q_2^2+2 p_1 p_2q_2+ p_2^2q_1\right)}{k}.
\end{array}
\end{equation}
Each of these solutions then yields a KE example on the KS orbifold  $(S_\bfn,\grD_\bfm)=(\bbp(\BOne\oplus \calo(n_1,n_2)),\grD_\bfm )$ with fibers
$\bbc\bbp^1[m_0/\gcd(m_0,m_\infty),m_\infty/\gcd(m_0,m_\infty)]/\bbz_{\gcd(m_0,m_\infty)}$.

\begin{proposition}\label{KEexistencethm}
There exists a four-parameter family of KS orbifolds with KE orbifold metrics. The parameters $(p_1,p_2,q_1,q_2)$ are integers such that
$0<p_1<q_1$, $0<|p_2|<q_2$, and $\gcd(|p_i|,q_i)=1$.
\end{proposition}

\begin{remark}
One might ask the following question. Suppose we have a co-prime quadruple $(n_1,n_2,m_0,m_\infty)$ satisfying the conditions in Proposition \ref{KEDioPhan}, that is,
$(S_\bfn,\Delta_\bfm)$ admits a KE metric with $\bfn=(n_1,n_2)$ and $\bfm=(m_0,m_\infty)$. Fixing $\bfn=(n_1,n_2)$, does there exist another pair 
$\tilde{\bfm} = (\tilde{m}^0,\tilde{m}^\infty)$ with $\tilde{\bfm} \neq \bfm$, such that $(n_1,n_2,\tilde{m}^0,\tilde{m}^\infty)$ is coprime and also satisfies the conditions in Proposition \ref{KEDioPhan}. In other words, if $\bfn$ has an appropriate choice of $\bfm$ such that $(S_\bfn,\Delta_\bfm)$ admits a KE metric, is this $\bfm$ then unique?
\end{remark}

\begin{example}\label{KEex}
If we pick $p_1=p_2=1$, $q_1=2$, and $q_2=q>2$, then \eqref{solnfam2} becomes
$$
\begin{array}{ccccl}
n_1&=&\frac{4(5q -1)(7q+3) \left(3q^2+ q+ 1\right)}{k}&=&\frac{2(5q -1)(7q+3) \left(3q^2+ q+ 1\right)}{\hat{k}}\\
n_2&=&\frac{2(5q -1) (7q+3) \left(13q+4\right)}{k}&=&\frac{(5q -1) (7q+3) \left(13q+4\right)}{\hat{k}}\\
m_0&=&\frac{2(7q+3) \left(13q+4\right) \left(3q^2+ q+ 1\right)}{k}&=&\frac{(7q+3) \left(13q+4\right) \left(3q^2+ q+ 1\right)}{\hat{k}}\\
m_\infty&=&\frac{2(5q -1) \left(13q+4\right) \left(3q^2+q+ 1\right)}{k}&=&\frac{(5q -1) \left(13q+4\right) \left(3q^2+q+ 1\right)}{\hat{k}},
\end{array}
$$
where 
 $$
 \hat{k} =\gcd \left(\begin{array}{r}
 2(5q -1)(7q+3) \left(3q^2+ q+ 1\right),\\
(5q -1) (7q+3) \left(13q+4\right),\\
(7q+3) \left(13q+4\right) \left(3q^2+ q+ 1\right),\\
(5q -1) \left(13q+4\right) \left(3q^2+q+ 1\right) 
\end{array}\right).$$
Likewise, if we pick $p_1=1$, $p_2=-1$, $q_1=2$, and $q_2=q>2$, then \eqref{solnfam2} becomes
$$
\begin{array}{ccccl}
n_1&=&\frac{4(5q +1)(7q-3) \left(3q^2-q+ 1\right)}{k}&=&\frac{2(5q +1)(7q-3) \left(3q^2-q+ 1\right)}{\hat{k}}\\
n_2&=&\frac{-2(5q +1) (7 q-3) \left(13q-4\right)}{k}&=&\frac{-(5q +1) (7 q-3) \left(13q-4\right)}{\hat{k}}\\
m_0&=&\frac{2(7 q-3) \left(13q-4\right) \left(3q^2- q+ 1\right)}{k}&=&\frac{(7 q-3) \left(13q-4\right) \left(3q^2- q+ 1\right)}{\hat{k}}\\
m_\infty&=&\frac{2(5q +1) \left(13q-4\right) \left(3q^2-q+1\right)}{k}&=&\frac{(5q +1) \left(13q-4\right) \left(3q^2-q+1\right)}{\hat{k}},
\end{array}
$$
where 
 $$\hat{k} =\gcd \left(\begin{array}{r}
2(5q +1)(7q-3) \left(3q^2-q+ 1\right),\\
(5q +1) (7 q-3) \left(13q-4\right),\\
(7 q-3) \left(13q-4\right) \left(3q^2- q+ 1\right),\\
(5q +1) \left(13q-4\right) \left(3q^2-q+1\right) 
\end{array}\right).$$
The appendix contains a table representing a sample family of these solutions as well as the classic smooth Koiso-Sakane example from Example \ref{ksvariations}.
Using \eqref{index}, we calculated the (orbifold) index $\cali_{\bfn,\bfm}$ of each $(S_\bfn,\grD_\bfm)$.

\end{example}

\subsection{Extremal and CSC K\"ahler metrics}
More generally, the admissible metrics are {\em extremal}, as defined by Calabi \cite{Cal82}, if and only if the scalar curvature of $g$, which is given as a function of $\gz$ by
\begin{equation}
Scal(g) =  
\frac{2 s_1 r_1 }{1+r_1\gz} +\frac{2 s_2 r_2 }{1+r_2\gz} - \frac{F''(\gz)}{p_c(\gz)},
\label{scalarcurvature}
\end{equation}
is a holomorphic potential, i.e., a linear affine function of $\gz$. Following the arguments in \cite{ACGT08} and considering the orbifold case at hand, it is easy to see that
the proof of Proposition 11 together with Section 2.2 in \cite{ACGT08} adapts to give us the following result.
\begin{proposition}\label{extremalexistence}
Any K\"ahler class $\Omega_\bfr$ on $(S_\bfn,\Delta_\bfm)$ admits an admissible extremal metric with scalar curvature
equal to an affine linear function of $\gz$. Moreover, this metric is (positive) CSC (i.e., the function is constant) if and only if
$\alpha_0 \beta_1-\alpha_1\beta_0 =0$, where
\begin{equation}\label{alpha-betaCSC}
\begin{split}
\alpha_r =& \int_{-1}^1  t^r p_c(t)dt \\
\beta_r   = &  \int_{-1}^1\Big(r_1 s_1(1+r_2t)+r_2 s_2(1+r_1t)\Big)  t^r dt \\
                        & + (-1)^r p_c(-1)/m_\infty +  p_c(1)/m_0.
                         \end{split}
                         \end{equation}
\end{proposition}
In the smooth case, this has been thoroughly explored in \cite{Gua95}, \cite{Hwa94}, and \cite{HwaSi02}, as well as Section 3.4 of \cite{ACGT08}. Here we will just mention that with $s_1=2/n_1$ and $s_2=2/n_2$, the equation
$\alpha_0 \beta_1-\alpha_1\beta_0 =0$ is equivalent to $f(r_1,r_2)=0$ where
\begin{equation}\label{CSCcond}
\begin{split}
f(r_1,r_2)= & 9 (m_0 - m_\infty) n_1 n_2 - 6 (m_0 + m_\infty) n_1 n_2 (r_1 + r_2) + 6 (m_0 - m_\infty) n_1 n_2 r_1r_2\\
&+ 3n_2 (4 m_0 m_\infty - n_1 (m_0 - m_\infty)) r_1^2 +
 3n_1 (4 m_0 m_\infty - n_2 (m_0 - m_\infty))  r_2^2 \\
 & - (4 m_0 m_\infty (n_1 + n_2) - 
    3 (m_0 - m_\infty) n_1 n_2) r_1^2 r_2^2.
    \end{split}
\end{equation}

\begin{proposition}\label{orbtocsc}
For any value of $n_1, n_2 \neq 0$, there exist a choice of $m_0$ and $m_\infty$ such that $(S_\bfn,\Delta_\bfm)$ admits K\"ahler classes $\Omega_\bfr$ with admissible constant scalar curvature K\"ahler metrics.
\end{proposition}

\begin{proof}
In the case of $n_1n_2<0$ we can choose $m_0=m_\infty$ and the existence of CSC K\"ahler classes follows from \cite{Hwa94} (see also \cite{ACGT08} for details using the present notation). In the case of $n_1n_2>0$, notice that $f(0,0)=9(m_0-m_\infty)n_1n_2$ while $f(1,1)=8 m_\infty (m_0(n_1 +n_2) - 3 n_1 n_2)$. So, if we choose
$m_0$ and $m_\infty$ such that $m_\infty > m_0> \frac{3n_1 n_2}{n_1+n_2}$, we have $f(0,0)<0$ and $f(1,1)>0$. This means that any continuous curve going from
$(0,0)$ to $(1,1)$ in the square $0<r_1,r_2<1$ must contain at least one point $(r_1,r_2)$ where $f$ vanishes. This ensures the existence of solutions $0<r_1,r_2<1$ to 
the equation $f(r_1,r_2)=0$. In turn, this implies the existence of classes $\Omega_\bfr$ with admissible constant scalar curvature K\"ahler metrics.
\end{proof}

\begin{example}\label{KoisoSakaneClassic}
When $m_0=m_\infty=1$, $n_1=1$, and $n_2=-1$, we have $f(r_1,r_2) = -12(-1+r_1-r_2)(r_1+r_2)$, reconfirming the smooth CSC K\"ahler metrics on \newline
$\bbp(\BOne \oplus \calo(1,-1)\ra{1.8} \bbc\bbp^1\times\bbc\bbp^1$ as stated in Theorem 9 of \cite{ACGT08}.
\end{example}

\begin{remark}\label{4.14caution}
Note that in general there is no guarantee the $CSC$ K\"ahler classes $\Omega_\bfr$ from Proposition \ref{orbtocsc} are rational (i.e. both $r_1$ and $r_2$ are rational) and thus useful from the Sasakian geometry point of view. With the constraint that $r_1, r_2$ are rational, the equation $f(r_1,r_2)=0$ is diophantine in nature. Suppose $f(r_1,r_2) \neq 0$ for some rational class $\Omega_\bfr$ and assume we have a Boothby-Wang constructed Sasaki manifold defined by
an integer orbifold K\"ahler class obtained by an appropriate rescale of $\Omega_r$. We shall see in the next subsection that then we will have a CSC Sasaki metric 
somewhere else in the Sasaki cone of this Sasaki maniifold. The key to seeing this is to use the connection discovered by Apostolov and Calderbank in \cite{ApCa18} between the so-called {\em weighted extremal metrics} and extremal metrics appearing as transverse structures for different rays in the same Sasaki cone.
\end{remark}

\subsection{Weighted Extremal Metrics}
Let $(M,g,\omega)$ be a K\"ahler orbifold of complex dimension $m$, $f$ a positive Killing potential on $M$, and (weight) $p\in \bbr$.
Then the  $(f,p)$-Scalar curvature of $g$ is given by
\begin{equation}\label{scalh1}
Scal_{f,p}(g)= f^2 Scal(g)  -2(p-1) f\Delta_g f - p(p-1)|df|^2_g,
\end{equation}
If $Scal_{f,p}(g)$ is a Killing potential, $g$ is said to be a {\it $(f,p)$-extremal K\"ahler metric}.  
The case $p=2m$ has been studied by several people and is interesting due to the fact that $Scal_{f,2m}(g)$ computes the scalar curvature of the Hermitian metric $h = f^{-2} g$.

However, the case of interest to us here is when $p=m+2$. This case is related to the study of extremal Sasaki metrics \cite{ApCa18, ApCaLe21, ApJuLa21}. Indeed, if we assume that the K\"ahler class $[\omega/2\pi]$ is an integer orbifold class giving a Boothby-Wang constructed (smooth) Sasaki manifold over $(M,g,\omega)$, then $\frac{Scal_{f,m+2}(g)}{f}$ is equal to the transverse scalar curvature of a certain Sasaki structure (determined by $f$) in the Sasaki cone. More precisely, if $\chi$ is the Reeb vector field of the Sasaki structure coming directly from the Boothby-Wang construction over $(M,g,\omega)$ and $f$ is viewed as a pull-back to the Sasaki manifold, then (mod $\cald$) $\xi:= f\chi$ is a Reeb vector field in the Sasaki cone giving a new Sasaki structure. While the pull-back from $M$ of $Scal(g)$ is the Tanaka-Webster scalar curvature of the Tanaka-Webster connection induced by $\chi$, the expression $\frac{Scal_{f,m+2}(g)}{f}$ pulls back from $M$ to be the Tanaka-Webster scalar curvature of the Tanaka-Webster connection induced by $\xi$. The latter is then also identified with the transverse scalar curvature of the Sasaki structure defined by $\xi$. This fact is seen from the details of the proof of Lemma 3 in \cite{ApCa18}. As also follows from \cite{ApCa18} (see their Theorem 1), the Sasaki structure determined by $f$ is extremal if and only if $g$ is $(f,m+2)$-extremal.

Now, returning to our orbifolds at hand we have that $m=3$ and so $m+2=5$. Note that $|\gz+b|$ for $b\in \bbr$ such that $|b|>1$ defines a Killing potential on $(S_\bfn,\Delta_\bfm)$. 

It is not hard to check that Section 2 and the existence result in the second half of Theorem 3.1. of \cite{ApMaTF18} adapts to get us the next proposition.
Note in particular, that whether we use $f=\gz+b$ (when ($b>1$) or $f=-(\gz+b)$ (when $b<-1$), the formula for $Scal_{f,p}$ in \eqref{scalh1} will give us the same result, i.e., the right hand side of (19) in \cite{ApMaTF18}. In general the assumption of $b>1$ in \cite{ApMaTF18} is merely practical and all the arguments are easily adapted to include the $b<-1$ case as well. Further, the root counting argument in the proof of Theorem 3.1.in \cite{ApMaTF18} is not affected by our mild orbifold conditions.

\begin{proposition}\label{CRtwistCSC} 
Let $b\in \bbr$ such that $|b|>1$. 
Any K\"ahler class $\Omega_\bfr$ on $(S_\bfn,\Delta_\bfm)$ admits a $(|\gz +b|,5)$-extremal K\"ahler metric with weighted scalar curvature
$Scal_{|\gz +b|,5}=A_1\gz+A_2$ for constants $A_1, A_2$ given as the unique solutions of the following linear system
\begin{equation}\label{A1andA2}
\begin{array}{ccc}
\alpha_{1, -6} A_1 + \alpha_{0, -6} A_2 =  2 \beta_{0, -4}\\
\\
\alpha_{2, -6} A_1 + \alpha_{1, -6} A_2 =  2\beta_{1, -4},
\end{array}
\end{equation}
where
\begin{equation}\label{alpha-beta-r-q}
\begin{split}
\alpha_{r,-6} =& \int_{-1}^1 (t+ b)^{-6} t^r p_c(t)dt \\
\beta_{r,-4}   = &  \int_{-1}^1\Big( \frac{2r_1}{n_1}(1+r_2 t)+ \frac{2r_2}{n_2}(1+r_1 t)\Big)  t^r (t+ b)^{-4}dt \\
                        & +(-1)^r(b-1)^{-4}p_c(-1)/m_\infty + (1+b)^{-4} p_c(1)/m_0.
                         \end{split}
                         \end{equation}
Moreover, assuming that $\Omega_\bfr$ is rational and we can form a Boothby-Wang constructed Sasaki manifold with respect to an appropriate rescale of
$\Omega_\bfr$, the extremal Sasaki structure determined by $f=|\gz+b|$ has constant scalar curvature if and only if 
$Scal_{|\gz +b|,5}=A_1\gz+A_2$ is a constant multiple of $|\gz+b|$, i.e., if and only if  $A_1 b-A_2=0$.                        
\end{proposition}
Suppose a rational $\Omega_\bfr$ is given on a specific $(S_\bfn,\Delta_\bfm)$.
Solving the linear system for $A_1$ and $A_2$ in Proposition \ref{CRtwistCSC} (note that $\alpha_{1,-6}^2-\alpha_{0,-6}\alpha_{2,-6} \neq 0$) and simplifying a bit, the equation $A_1 b-A_2=0$ may be re-written as
$h(b)=0$ where
\begin{equation}\label{polh}
h(b)= (b^2-1)^7(b(\alpha_{1,-6}\beta_{0,-4}-\alpha_{0,-6}\beta_{1,-4})-(\alpha_{1,-6}\beta_{1,-4}-\alpha_{2,-6}\beta_{0,-4})).
\end{equation}
Using \eqref{alpha-beta-r-q} we see that $h(b)$ simplifies as a polynomial of degree $5$ in the variable $b$ with leading coefficient equal to
$\frac{2f(r_1,r_2)}{9m_0m_\infty n_1n_2}$, where $f(r_1,r_2)$ is given by \eqref{CSCcond}. Since further $\lim_{b\rightarrow \pm 1^{\mp}}h(b) >0$, we conclude that unless $f(r_1,r_2)=0$ (in which case $\Omega_\bfr$ itself has an admissible CSC representative), there must exist
at least one value $b\in (-\infty,-1)\cup (1,+\infty)$ such that $h(b)=0$, hence $A_1 b-A_2=0$, and therefore, for this $b$ value, the extremal Sasaki structure determined by $f=|\gz+b|$ has constant scalar curvature. If $b$ is an irrational number, then the corresponding Sasaki structure is irregular.

\begin{theorem}\label{CSCSasakiExistence}
Suppose $\Omega_\bfr$ is a rational admissible K\"ahler class on a KS orbifold of the form $(S_\bfn,\Delta_\bfm)$.
Let $\cals$ be a Boothby-Wang constructed  Sasaki manifold given by an appropriate rescale of $\Omega_\bfr$. Then the corresponding Sasaki cone will always have a
(possibly irregular) CSC-ray (up to isotopy). 
\end{theorem}

\begin{remark}\label{cscremark}
It is not too difficult (using \eqref{scalarcurvature} and \eqref{positivityForbifold}) to produce some admissible K\"ahler metric of positive scalar curvature in $\Omega_\bfr$. Using the corresponding K\"ahler form of this metric for the Boothby-Wang construction, we can use Lemma 5.2 in \cite{BHL17}) to conclude that the CSC-ray alluded to in the theorem above has positive transverse scalar curvature. This together with Theorem \ref{CSCSasakiExistence} proves the main theorem in the introduction.
\end{remark}

\begin{remark}
The observations in this section - in particular, Proposition \ref{CRtwistCSC} and Theorem \ref{CSCSasakiExistence} - can be generalized to no-blown-down admissible manifolds in general \cite{BHLT21}.
Indeed, in \cite{BHLT21} we use a more general version of Proposition \ref{CRtwistCSC} to answer the open question in Problem 6.1.2 of \cite{BHLT19} by providing a counter example where the Sasaki cone contains some extremal Sasaki structures, but no CSC Sasaki metric at all.
 \end{remark}

\begin{example}
Suppose $m_1=m_2=1$ and suppose $n_1> 4$. Now let $r_1=\frac{5 n_1^2-4}{n_1 \left(n_1^2+4\right)}$, $r_2=2/n_1$, and $n_2$ be equal to any positive integer.
Note that since $n_1> 4$ we have ensured that $r_i \in (0,1)$ for $i=1,2$ and $r_1 \neq r_2$.
Then $\Omega_\bfr$ is a rational (non-diagonally) admissible K\"ahler class on the smooth KS manifold $(S_\bfn)$. Note that since $\Omega_\bfr$ is non-diagonally admissible, we know from Theorem \ref{whenS3join} that $S_\bfn$  (with any appropriate rescale of $\Omega_\bfr$) is not the quotient of a regular ray in the $\bfw$-cone of the type of $S^3_\bfw$-join
described in Theorem \ref{whenS3join}.
Further, it can be shown (as we will do near the end of this example) that $S_\bfn$  (with any appropriate rescale of $\Omega_\bfr$) is also not the quotient of the regular ray in the $\gt^+_{sphr}$ cone of a Yamazaki fiber join over $\bbc\bbp^1\times\bbc\bbp^1$. The Boothby-Wang constructed  Sasaki manifold over $S_\bfn$, given by an appropriate rescale of $\Omega_\bfr$ is thus in a certain sense a truly new example.

By Theorem 8 and Proposition 11 of \cite{ACGT08} we know that the K\"ahler class $\Omega_\bfr$, while having an extremal admissible K\"ahler metric representative, does not admit a CSC K\"ahler metric. [This fact can also easily be checked by seeing that $f(r_1,r_2)$ from \eqref{CSCcond} cannot be zero in this case.] 
We now calculate that with the given choices of $(\bfm, \bfn, \bfr)$, we have that $h(b)$ from \eqref{polh} is equal to
$$
h(b)=\frac{4 (n_1-2 b)p(b) }{9 n_1^5 \left(n_1^2+4\right)^2 n_2}
$$
where $p(b)$ is given by
$$
\begin{array}{ccl}
p(b) &=&3 n_1^8 n_2-4 n_1^7-13 n_1^6 n_2+268 n_1^5+1252 n_1^4 n_2-544 n_1^3-1456 n_1^2 n_2+192 n_1+192 n_2\\
\\
&-& 4 n_1 \left(9 n_1^6 n_2+20 n_1^5+482 n_1^4 n_2+64 n_1^3-464 n_1^2 n_2-64 n_1+224 n_2\right)b\\
\\
&+& 16 \left(n_1^7+79 n_1^6 n_2-17 n_1^5-35 n_1^4 n_2+56 n_1^3+32 n_1^2 n_2-16 n_1-16 n_2\right)b^2\\
\\
&-&4 n_1 \left(63 n_1^6 n_2-20 n_1^5+106 n_1^4 n_2-64 n_1^3-16 n_1^2 n_2+64 n_1-32 n_2\right)b^3\\
\\
&+&(21 n_1^8 n_2-12 n_1^7+21 n_1^6 n_2+4 n_1^5+268 n_1^4 n_2-352 n_1^3-208 n_1^2 n_2+64 n_1+64 n_2)b^4.
\end{array}
$$
Clearly, $b=n_1/2$ is a rational root of $h(b)$. Note that since $n_1>4$ and $n_2>0$, we also have that 
\begin{itemize}
\item The $b^4$ coefficient of $p(b)$ is positive
\item $p(\pm 1)>0$ 
\item $p'(-1)<0$ and $p'(1)>0$
\item $f(b):= p''(b)$ is a concave up 2nd order polynomial with $f(\pm 1)>0$, $f'(-1)<0$, and $f'(1)>0$. Thus any roots of $f(b)$ would be
inside the interval $(-1,1)$.
\end{itemize}
Putting this together we have that $p(b)$ is positive and decreasing at $b=-1$, positive and increasing at $b=1$, and concave up for $|b|\geq 1$.
Hence $p(b)$ has no roots in $(-\infty, -1)\cup (1,+\infty)$ and hence $b=n_1/2$ is the ONLY root of $h(b)$.

All in all we conclude that while the original regular ray coming from the Boothby-Wang constructed Sasaki manifold over $S_\bfn$, given by an appropriate rescale of $\Omega_\bfr$ is not CSC, the CSC ray alluded to in Theorem \ref{CSCSasakiExistence} is quasi-regular. It would be interesting to explore what the transverse K\"ahler orbifold is for this ray. 
Conjecturely, it may be a KS orbifold.

Finally, let us finish this example by confirming that $S_\bfn$ (with any appropriate rescale of $\Omega_\bfr$) is also not the quotient of the regular ray in the $\gt^+_{sphr}$ cone of a Yamazaki fiber join over $\bbc\bbp^1\times\bbc\bbp^1$: 

For $S_\bfn$ (with an appropriate rescale of $\Omega_\bfr$) to be such a quotient, it follows from Section \ref{yam} and \eqref{yamequation} that
there should exist $k_1^1, k_1^2, k_2^1, k_2^2 \in \bbz^+$ such that
$$k_1^1-k_2^1 =  n_1,\quad k_1^2-k_2^2=n_2,\quad \frac{k_1^1-k_2^1}{k_1^1+k_2^1}=\frac{5 n_1^2-4}{n_1 \left(n_1^2+4\right)},\quad \frac{k_1^2-k_2^2}{k_1^2+k_2^2} = 2/n_1.$$
This would place further restrictions on $n_1$ and $n_2$ and, in particular, we would need $n_1$ to satisfy that
$\frac{n_1^2 \left(n_1^2+4\right)}{5 n_1^2-4}$ is a positive  integer (that positive integer is $k_1^1+k_2^1$). By the following elementary lemma this is not possible and
so $S_\bfn$ (with any appropriate rescale of $\Omega_\bfr$) is not the quotient of the regular ray in the $\gt^+_{sphr}$ cone of a Yamazaki fiber join.

\begin{lemma}
For all integers $x \geq 3$, $\frac{x^2 \left(x^2+4\right)}{5 x^2-4}$ is NOT an integer.
\end{lemma}

\begin{proof}
Assume for contradiction that there exists an integer $x\geq 3$ such that $\frac{x^2 \left(x^2+4\right)}{5 x^2-4}$ is an integer.
Let $d>1$ be a prime divisor of $5 x^2-4$. Then, due to our assumption, $d \mid x^2$ or $d\mid(x^2+4)$. If $d\mid x^2$, then, since also $d\mid (5 x^2-4)$, we have that
$d\mid4$ and so $d=2$. If $d\mid(x^2+4)$, then, since again $d\mid (5 x^2-4)$, we have that
$d\mid (5(x^2+4) - (5 x^2-4))$, i.e., $d\mid 24$, and so $d=2$ or $d=3$. 

By considering the three possible cases $x=3k$, $x=3k+1$, and $x=3k+2$ with $k\in \bbz$, we realize that $3 \nmid (5 x^2-4)$.
Therefore, $d=2$ and  $5x^2-4$ is even. This implies that $x^2$ is even and hence $x$ is even. In particular, $x\neq 3$ and moving forward we may assume
$x\geq 4$.

Since $5x^2-4$ is even, we have that $(5 x^2-4)=2^i$ for some non-negative integer $i$. We also know that $5x^2-4 \geq 76$ (since $x\geq 4$), so $i>6$.

Now, writing $x=2k$ for some $k\in \bbz^+$ we see that $(5 x^2-4)=2^i$ implies that
$20k^2-4=2^i$ and hence $5k^2=2^{i-2}+1$. Thus $5k^2$ is odd, which gives us that $k^2$ is odd and hence $k$ is odd.
Let us write $k=2l+1$ for some $l\in \bbz^+$. So $x=2(2l+1)$. Then 
$$
\begin{array}{ccl}
x^2 \left(x^2+4\right)&=& 2^2(4l^2+4l+1)(2^2(4l^2+4l+1)+4)\\
\\
&=& 2^2(4l^2+4l+1)(2^2(4l^2+4l+2))\\
\\
&=& 2^2(4l^2+4l+1)(2^3(2l^2+2l+1))\\
\\
&=& 2^5(2(2l^2+2l)+1)(2(l^2+l)+1).
\end{array}
$$
Since $(2(2l^2+2l)+1)(2(l^2+l)+1)$ is odd, we realize that $\frac{x^2 \left(x^2+4\right)}{5 x^2-4}= \frac{(2(2l^2+2l)+1)(2(l^2+l)+1)}{2^{i-5}}$. Since $i>6$, this cannot possible be an integer and hence we have arrived at a contraction. This completes the proof of the lemma.
\end{proof}

\end{example}

As an immediate consequence of Theorem \ref{CSCSasakiExistence} we get the following special case of the existence result for toric Sasaki-Einstein metrics due to Futaki, Ono, and Wang \cite{FOW06}.

\begin{corollary}\label{FOWspecial}
Suppose a KS orbifold of the form $(S_\bfn,\Delta_\bfm)$ is log Fano and consider a Boothby-Wang constructed  Sasaki manifold over $(S_\bfn,\Delta_\bfm)$ given by some K\"ahler form representing $c_1^{orb}(S_\bfn,\grD_\bfm)/\cali_{\bfn,\bfm}$. Then the corresponding Sasaki cone will always have a
(possibly irregular) Sasaki-Einstein structure (up to isotopy).
\end{corollary}

\begin{proof}
First, we note that $c_1(\cald)=0$ and so the CSC Sasaki metrics from Theorem \ref{CSCSasakiExistence} above are now $\eta$-Einstein. Second, we see from Proposition 5.3 in \cite{BHL17}) that
since the basic first Chern class of the initial Sasaki metric is positive (as a pullback of the positive class $c_1^{orb}(S_\bfn,\grD_\bfm)/\cali_{\bfn,\bfm}$), the average transverse scalar curvature of any Sasaki structure in the Sasaki cone must be positive. In particular, the transverse (constant) scalar curvature of any $\eta$-Einstein structure in the cone must be positive.
This means that any $\eta$-Einstein ray admits a Sasaki-Einstein structure.
\end{proof}

\begin{remark}\label{MNapproach}
Examples similar in spirit to Corollary \ref{FOWspecial}, but also including non-toric cases have been given by Mabuchi and Nakagawa \cite{MaNa13}. See also \cite{ApJuLa21} and references therein.
\end{remark}

\section*{Appendix}

\bigskip

\begin{tabular}{|c|c|c|c|c|c|c|c|c|c|c|c|} 
 \hline
$p_1$&$q_1$&$p_2$& $q_2$ &$n_1$&$n_2$& $m_0$ & $m_\infty$ & $m$ &$v_0$&$v_\infty$& $\cali_{\bfn,\bfm}$\\ [0.5ex] 
  \hline
 1&2&-1&15&5124072&-740316&6438801&4797538&126251&51&38&89\\ 					
 \hline
 1&2&-1&14&775675&-120061&972325&726685&10235&95&71&83\\ 					
 \hline
 1&2&-1&13&48&-8&60&45&15&4&3&7\\ 					
  \hline
 1&2&-1&12&2080161&-375516&2591676&1951756&31996&81&61&71\\ 						
  \hline
 1&2&-1&11&1462832&-288008&1815479&1373876	&49067&37&28&65\\ 					
  \hline
 1&2&-1&10&110483&-23919&136479&103887&2037	&67&51&59\\ 					
 \hline
 1&2&-1&9&129720&-31188&159330&122153&5311&30&23&53\\ 						
 \hline
 1&2&-1&8&80401&-21730&98050&75850&1850&53&41&47\\ 						
  \hline
 1&2&-1&7&25944&-8004&31349&24534&1363&23&18&41\\ 
  \hline
 1&2&-1&6&124527&-44733&148629&118141&3811&39&31&35\\ 					
 \hline
 1&2&-1&5&59072&-25376&69296&56303&4331&16&13&29\\ 
 \hline
 1&2&-1&4&525&-280&600&504&24 &25&21&23\\
 \hline
1&2&-1&3&1440&-1008&1575&1400&175&9&8&17\\
 \hline
1&2&-1&2&1&-1&1&1&1&1&1&1\\
 \hline
1&2&1&2&51&51&85&45&5&17&9&13\\
\hline
1&2&1&3&10416&7224&15996&9331&1333&12&7&19\\
\hline
1&2&1&4&31217&16492&46004&28196&1484&31&19&25\\
\hline
1&2&1&5 &8208&3496&11799&7452&621&19&12&31\\ 
\hline
1&2&1&6 &30015	&10701&42435&27347&943&45&29&37\\ 					
\hline
1&2&1&7 &54808&16796&76570&50065&2945&26&17&43\\ 					
\hline
1&2&1&8 &51389&13806&71154&47034&1206&59&39&49\\ 					
\hline
1&2&1&9 &552&132	&759&506&253&3&2&5\\ 					
\hline
1&2&1&10 &1112447&239659&1521101&1021013&20837&73&49&61\\ 				
\hline
1&2&1&11 &36000&7056&49000&33075&1225&40&27&67\\ 				
\hline
1&2&1&12 &456837&82128&619440&420080&7120&87&59&73\\ 					
\hline
1&2&1&13 &3134336&520384&4236251&2884256&90133&47&32&79\\ 						
\hline
1&2&1&14 &466923&72013&629331&429939&6231&101&69&85\\ 					
\hline
1&2&1&15 &5522472&795204&7425486&5087833&137509&54&37&91\\ [1ex] 					
 \hline
\end{tabular}
\normalsize
\medskip
\newline TABLE: This gives a sample of K\"ahler-Einstein orbifold solutions. See Example \ref{KEex}.

\newpage
\newcommand{\etalchar}[1]{$^{#1}$}
\def\cprime{$'$} \def\cprime{$'$} \def\cprime{$'$} \def\cprime{$'$}
  \def\cprime{$'$} \def\cprime{$'$} \def\cprime{$'$} \def\cprime{$'$}
  \def\cdprime{$''$} \def\cprime{$'$} \def\cprime{$'$} \def\cprime{$'$}
  \def\cprime{$'$}
\providecommand{\bysame}{\leavevmode\hbox to3em{\hrulefill}\thinspace}
\providecommand{\MR}{\relax\ifhmode\unskip\space\fi MR }
\providecommand{\MRhref}[2]{%
  \href{http://www.ams.org/mathscinet-getitem?mr=#1}{#2}
}
\providecommand{\href}[2]{#2}

\end{document}